\documentclass[12pt]{amsart}
\usepackage{amsfonts}
\usepackage{amsmath}
\usepackage{amssymb}
\usepackage{graphics}
\usepackage[pagebackref]{hyperref}
\usepackage{amsmath,amscd}

\newtheorem{theorem}{Theorem}[section]
\newtheorem{corollary}[theorem]{Corollary}
\newtheorem{lemma}[theorem]{Lemma}

\newtheorem{lemma and definition}[theorem]{Lemma and Definition}
\newtheorem{proposition}[theorem]{Proposition}
\newtheorem{definition}[theorem]{Definition}
\newtheorem{notation}[theorem]{Notation}
\newtheorem{remark}[theorem]{Remark}
\newtheorem{example}[theorem]{Example}
\newtheorem{discussion}[theorem]{Discussion}

\newtheorem{the construction}[theorem]{THE CONSTRUCTION}
\newcommand{\field}[1]{\mathbb{#1}}
\newcommand{\C}{\field{C}}
\newcommand{\R}{\field{R}}
\newcommand{\Q }{\field{Q}}
\newcommand{\Z }{\field{Z}}
\newcommand{\N }{\field{N}}
\DeclareMathOperator{\htt}{ht} \DeclareMathOperator{\qf}{qf}
\DeclareMathOperator{\Ann}{Ann} \DeclareMathOperator{\td}{t.d.}
\DeclareMathOperator{\Ker}{Ker} \DeclareMathOperator{\Ima}{Im}
\DeclareMathOperator{\Spec}{Spec} \DeclareMathOperator{\Max}{Max}
\DeclareMathOperator{\Cl}{Cl} \DeclareMathOperator{\Pic}{Pic}
\DeclareMathOperator{\Supp}{Supp} \DeclareMathOperator{\rad}{rad}
\DeclareMathOperator{\Sup}{sup} \DeclareMathOperator{\Jac}{J}
\DeclareMathOperator{\Cok}{Cok}\DeclareMathOperator{\depth}{depth}
\renewcommand{\thefootnote}{\fnsymbol{footnote}}
\def\proof{{\parindent0pt {\bf Proof.}}}
\def\endproof{{\hfill $\Box$}}
\def\1{{\rm (1)}}
\def\2{{\rm (2)}}
\def\3{{\rm (3)}}
\def\4{{\rm (4)}}
\def\5{{\rm (5)}}
\def\6{{\rm (6)}}
\def\i{{\rm (i) }}
\def\ii{{\rm (ii) }}
\def\iii{{\rm (iii) }}
\def\iv{{\rm (iv) }}
\def\v{{\rm (v) }}

\begin{document}

\title{FFRT properties of hypersurfaces and their $F$-signatures}

\author{Khaled Alhazmy}
\email{khazmy@kku.edu.sa}
\address{Department of  Mathematics, College of Science,
King Khalid University, Main campus,  Abha 61413, Saudi Arabia}

\author{Mordechai Katzman}
\email{m.katzman@sheffield.ac.uk}
\address{School of  Mathematics,
University of Sheffield, Hicks Building, Sheffield S3 7RH, United Kingdom}

\begin{abstract}
This paper studies properties of certain hypersurfaces  in prime characteristic:
we give a sufficient and necessary conditions for some classes of such hypersurfaces to
have Finite $F$-representation Type (FFRT) and we compute
the $F$-signatures of these hypersurfaces.
The main method used in this paper is based on finding  explicit matrix factorizations.
\end{abstract}

\maketitle

\section{Introduction}
\label{Section: Introduction}

For any commutative ring $R$ of prime characteristic $p$, $R$-module $M$, and $e\geq 0$  we can construct a new $R$-module
$F_*^e M$ with elements $\{ F_*^e m \,|\, m\in M\}$, abelian group structure $(F_*^e m_1) + (F_*^e m_2)=F_*^e (m_1+m_2)$ and $R$-module structure given by
$r F_*^e m = F_*^e r^{p^e} m$. In this paper we study  properties of the modules $F_*^e R$ for various hypersurfaces $R$.

Smith and Van den Bergh introduced in \cite{SV} the first property studied in this paper, namely \emph{Finite F-representation Type} (henceforth abbreviated \emph{FFRT})
which we describe in some detail in Section \ref{Section: Rings of finite F-representation type}.
We say that $R$ has FFRT if there exists a finite set of indecomposable $R$-modules $\mathcal{M}=\{M_1, \dots, M_s\}$ such that for all $e\geqslant 0$,
$F_*^e R$ is isomorphic to a direct sum of summands in $\mathcal{M}$. Rings with FFRT have very good properties, e.g.,
with some additional hypothesis their rings of differential operators are simple \cite[Theorem 4.2.1]{SV},
their Hilbert-Kunz multiplicities are rational (cf.~ \cite{S}),
tight closure commutes with localization in such rings (cf.~\cite{Y}),
local cohomology modules of $R$ have finitely many associated primes (\cite{DaoQuyAssPrimes}), and
$F$-jumping coefficients in $R$ form discrete sets (\cite{MR3211813}).

The main tool in this paper, developed in sections \ref{Section: Matrix Factorization} and \ref{Section: The presentation as a cokernel of a Matrix Factorization},
is an explicit presentation of $F_*^e R$ as cokernel of a matrix factorization.
In section \ref{Section: When does f+uv have finite F-representation type?}
we classify those $F$-finite hypersufaces
of the form

$K[\![x_1, \dots, x_n, u,v]\!]/(f(x_1, \dots, x_n)+uv)$
which have FFRT in terms of the hypersurfaces defined by the powers of $f$. As a corollary we construct in section \ref{Section: A Class of rings that have FFRT but not finite CM type} examples of rings that have FFRT but not finite CM representation type.

The last two sections of this paper analyze the \emph{$F$-signatures} of some hypersurfaces. We recall the following definition.

\begin{definition}
Let   $(R,\mathfrak{m},K)$ be a $d$-dimensional  $F$-finite Noetherian local ring of prime characteristic
$p$. If $[K:K^p]<\infty$  is the dimension of $K$ as $K^p$-vector space and  $\alpha (R) =\log_p[K:K^p]$, then the \emph{$F$-signature of $R$}, denoted by $\mathbb{S}(R)$,  is defined as
\begin{equation}\label{eqn: F-signature}
 \mathbb{S}(R)=\lim_{e\rightarrow \infty} \frac{\sharp (F_*^e(R),R)}{p^{e(d+\alpha (R))}}
\end{equation}
where $\sharp (F_*^e(R),R)$ denotes the maximal rank of a free direct summand of $F_*^e(R)$.
\end{definition}

$F$-signatures were first defined by C.~Huneke and G.~Leuschke in \cite{HL} for
$F$-finite local rings of prime
characteristic with a perfect residue field.
In \cite{Y2} Y.~Yao extended the definition of $F$-signature to arbitrary local rings without the assumptions that the residue field be perfect and
recently K.~Tucker  proved in  \cite{KT} that the limit in (\ref{eqn: F-signature}) always exists.

The $F$-signature  seems to give subtle information on the singularities of  $R$. For example, the $F$-signature of any of the two-dimensional
quotient singularities ($A_n$), ($D_n$), ($E_6$), ($E_7$), ($E_8$) is the reciprocal of the order of the group
$G$ defining the singularity \cite[Example 18]{HL}. In \cite{AL}   I.~Aberbach and G.~Leuschke proved  that the $F$-signature is positive if
and only if $R$ is strongly F-regular. Furthermore, we have $\mathbb{S}(R) \leq 1$ with equality if and only if R is regular (cf.~\cite[Theorem 4.16]{KT} and \cite[Corollary 16]{HL}).

In sections \ref{Section: The F-signature of f+uv when f is a monomial} and \ref{Section: The F-signature of f+z^2 when f is a monomial}
we apply the methods developed in sections \ref{Section: Matrix Factorization} and \ref{Section: The presentation as a cokernel of a Matrix Factorization}
in order to find explicit
expressions for the $F$-signatures of the rings
$K[\![x_1, \dots, x_n,u,v]\!]/(f(x_1, \dots, x_n)+uv)$ and $K[\![x_1, \dots, x_n,z]\!]/(f(x_1, \dots, x_n)+z^2)$ where $f$ is a monomial.

Throughout this paper, we shall assume all rings are commutative with a unit, Noetherian, and have prime characteristic $p > 0$ unless otherwise is stated.
We denote by $\mathbb{N} $ (and $\mathbb{Z}_{+}$) the set of the positive integers (and non-negative integers) respectively. We let $q=p^e$ for some $e \in \mathbb{N}$.

\section{Rings of finite F-representation type}
\label{Section: Rings of finite F-representation type}

Throughout this section  $R$ denotes a ring of prime characteristic $p$. In what follows we gather some properties of the modules $F_*^e(M)$ introduced in the previous section
and introduce several concepts related to FFRT.

Recall that a non-zero finitely generated module $M$ over a local ring $R$ is Cohen-Macaulay if $\depth_RM=\dim M$ and  $R$ is a Cohen-Macaulay ring if $R$ itself is  a Cohen-Macaulay module. However, if $\depth_RM=\dim R$, $M$ is called maximal Cohen-Macaulay module (or MCM module). When $M$ is a finitely generated module over non-local ring $R$, $M$ is Cohen-Macaualy if $M_\mathfrak{m}$ is a Cohen-Macaulay module for all maximal ideals $\mathfrak{m} \in \Supp M$.

Later in the paper we will look at the modules $F_*^e(R)$ when $R$ is a Cohen-Macaulay ring. Note that a regular sequence on an $R$-module $M$ is also a regular sequence
on $F_*^e M$, and in particular, if $R$ is Cohen-Macaulay, $F_*^e(R)$ are MCM modules for all $e\geq 0$.

\begin{definition}\label{D1}

A finitely generated $R$-module $M$ is said to have   finite F-representation type (henceforth abbreviated  FFRT) by finitely generated
$R$-modules $M_1, \dots, M_s$ if for every positive integer $e$, the $R$-module $F_*^e(M)$ can be written as a
finite direct sum in which each direct summand is isomorphic to some module in  $\{M_1,\ldots ,M_s\}$ , that is, there exist
non-negative integers $t_{(e,1)}, \ldots , t_{(e,s)}$ such that
 $$F_*^e(M)= \bigoplus_{j=1}^s M_j^{\oplus t_{(e,j)}} .$$
 We say that $M$ has FFRT if $M$ has FFRT by some finite set of $R$-modules $M_1,\ldots,M_s$
\end{definition}

If $W\subset R$ is a multiplicatively closed set and $I\subseteq R$ is an ideal, then for a finitely generated $R$-module $M$ one can check that $F_*^e(W^{-1}M)$  is isomorphic to
$W^{-1}F_*^e(M)$ as $W^{-1}R$-module and
$F_*^e(\widehat{M}_I)$ is isomorphic to $\widehat{F_*^e(M)}_I$ as $\hat{R}_I$-modules, where $\widehat{}_I$ denotes completion at $I$.

Examples of rings with FFRT can be found in \cite{SV}, \cite{TT} and  \cite{TS}. One such class of examples are \emph{rings of finite CM representation type}.
These are Cohen-Macaulay rings $R$ with the property that the set of all isomorphism classes of finitely generated, indecomposible, MCM $R$-modules is finite.
When $R$ is Cohen-Macaulay any direct summand of $F_*^e(R)$ is a MCM module, hence
finite Cohen-Macaulay representation type implies FFRT.
The converse is not true: we explore some examples of these in section \ref{Section: A Class of rings that have FFRT but not finite CM type}.


\section{Matrix Factorization}
\label{Section: Matrix Factorization}

In this section, we discuss the concept of a \textit{matrix factorization} and many basic properties that we need later in the rest of this paper.
We start by fixing some notations and stating some observations about the matrices and a cokernel of a matrix.

If  $m,n \in \mathbb{N}$, then  $M_{m \times n}(R)$ (and $M_{ n}(R)$)   denotes  the set  of all $m \times n$  (and $n \times n$)  matrices over a ring $R$. If $A \in M_{m\times n}(R)$ is the matrix representing the $R$-linear map  $\phi : R^{n}\longrightarrow R^{m}$  that is given by $ \phi (X)= AX$ for all $ X \in R^{ ^n}$ , then we write $A:R^{n}\longrightarrow R^{m}$ to denote the $R$-linear map $\phi$ and $\Cok_R(A)$ denotes the cokernel of $\phi$ while   $ \Ima_R(A)$ denotes the image of $\phi $. We  omit the subscript $R$ if it  is known from the context. If $A,B \in M_n(R)$, we say that $A$ is equivalent to $B$ if there exist invertible matrices $U,V \in M_n(R)$ such that $A =UBV$. If $A, B \in M_n(R)$ are equivalent matrices, then $\Cok_R(A)$ is isomorphic to  $\Cok_R(B)$ as $R$-modules. Furthermore, if $A \in M_n(R)$ and $B \in M_m(R)$, then  we define $A\oplus B$ to be  the matrix in $M_{m+n}(R)$ that is given by $A\oplus B = \left[
                                                                                                                     \begin{array}{cc}
                                                                                                                       A & 0 \\
                                                                                                                       0  & B \\
                                                                                                                     \end{array}
                                                                                                                   \right]$. In this case,  $ \Cok_R(A\oplus B)=\Cok_R(A)\oplus \Cok_R(B)$.

Let $\mathfrak{P}$ denote a ring with identity that is not necessarily commutative.
Let $m$ and $ n$ be  positive integers. If $\lambda \in \mathfrak{P}$, $ 1 \leq i \leq m $ and $ 1 \leq j \leq n  $, then $L_{i,j}^{m \times n}(\lambda)$ (and $L_{i,j}^{ n}(\lambda)$) denotes the $m \times n$ (and $n\times n$) matrix whose $(i,j)$ entry is $\lambda$ and the rest are all zeros. When $i\neq j$, we write $E_{i,j}^n(\lambda):= I_n+L_{i,j}^n(\lambda)$ where $I_n$ is the identity matrix in $M_n(R)$. If there is no ambiguity, we write $E_{i,j}(\lambda)$ (and $L_{i,j}(\lambda)$)  instead of $E_{i,j}^n(\lambda)$ (and $L_{i,j}^{ n}(\lambda)$).

The following lemmas describe an equivalence between specific matrices that are basic for our work in the rest of this paper.

\begin{lemma}\label{L2.9}
Let $m$ be an integer with $m \geq 2$  and $n=2m$. If $A$ is a matrix in $M_n(\mathfrak{P})$ that is given by
\begin{equation*}
  A= \left[
       \begin{array}{ccccccc}
         b &   &  &   &   &  & x  \\
         0 & b &  &   &   &   & \\
         1 & 0 & b &   &   &   &  \\
           & 1 & 0  & b &   &   &   \\
           &   & \ddots & \ddots & \ddots &   &  \\
           &   &   &  1 & 0 & b &   \\
           &   &   &   &  1 & 0 & b \\
       \end{array}
     \right]
\end{equation*}
then there exist two invertible matrices $M,N \in M_n(\mathfrak{P})$  satisfying that

 \begin{equation*}
   MAN=\left[
       \begin{array}{ccccccc}
         0 &   &  &   &   & (-1)^{m-1}b^m & x  \\
         0 & 0 &  &   &   &   &(-1)^{m-1}b^m \\
         1 & 0 & 0 &   &   &   &  \\
           & 1 & 0  & 0 &   &   &   \\
           &   & \ddots & \ddots & \ddots &   &  \\
           &   &   &  1 & 0 & 0 &   \\
           &   &   &   &  1 & 0 & 0 \\
       \end{array}
     \right]
 \end{equation*}

\end{lemma}

\begin{proof}
Use row and column operations  and the  induction on $m$.
\end{proof}

\begin{corollary}\label{C2.14}
Let $n=2m+1$ where $m$ is an integer with $m \geq 2$ and let $A$ be a matrix in $M_n(\mathfrak{P})$ given by
\begin{equation*}
  A= \left[
       \begin{array}{ccccccc}
         b &   &  &   &   & x & 0  \\
         0 & b &  &   &   &   & y \\
         1 & 0 & b &   &   &   &  \\
           & 1 & 0  & b &   &   &   \\
           &   & \ddots & \ddots & \ddots &   &  \\
           &   &   &  1 & 0 & b &   \\
           &   &   &   &  1 & 0 & b \\
       \end{array}
     \right]
\end{equation*}
Then there exist  invertible matrices $M$ and $N$ in $M_n(\mathfrak{P})$ such that
\begin{equation*}
 MAN =\left[
       \begin{array}{ccccccc}
         0 &   &  &   &   &  x &  (-1)^{m}b^{\frac{n+1}{2}} \\
         0 & 0 &  &   &   & (-1)^{m-1}b^{\frac{n-1}{2}}  & y \\
         1 & 0 & 0 &   &   &   &  \\
           & 1 & 0  & 0 &   &   &   \\
           &   & \ddots & \ddots & \ddots &   &  \\
           &   &   &  1 & 0 & 0 &   \\
           &   &   &   &  1 & 0 & 0 \\
       \end{array}
     \right]
\end{equation*}
\end{corollary}

\begin{proof}
Let $ \tilde{A}$ be the $2m \times 2m $ matrix given by

\begin{equation*}
\tilde{A}=\left[
       \begin{array}{ccccccc}
         b &   &  &   &   &  & x  \\
         0 & b &  &   &   &   & \\
         1 & 0 & b &   &   &   &  \\
           & 1 & 0  & b &   &   &   \\
           &   & \ddots & \ddots & \ddots &   &  \\
           &   &   &  1 & 0 & b &   \\
           &   &   &   &  1 & 0 & b \\
       \end{array}
     \right]
\end{equation*}
It follows that

 \begin{equation*}
   A= \left[
        \begin{array}{ccccc|c}
            &   &   &   &   & 0 \\
            &   &   &   &   & y \\
            &   &   \tilde{A} &   &   & 0 \\
            &   &   &   &   & \vdots \\
            &   &   &   &   & 0 \\ \hline
          0 & \ldots & 0 & 1 & 0 & b \\
        \end{array}
      \right]
 \end{equation*}

Now use Lemma  \ref{L2.9}  ,and  appropriate row and column operations to get the result.
\end{proof}

\begin{lemma}\label{L.16}
Let $n$ be an integer   with $n \geq 2 $. If $ A \in M_n(\mathfrak{P})$ is given by
 $$ A = \begin{bmatrix} b & & & & &  \\
                        1 & b & & & & \\
                          & 1 & b & & & \\
                          & & \ddots & \ddots & & \\
                         & & &1& b   \end{bmatrix}  $$
                         there exist upper triangular matrices $B, C \in M_n(\mathfrak{P})$ such that the $(i,i)$ entry of $B$ and $C$ is the identity element of $\mathfrak{P}$ for all $i=1, \dots, n $ and
                        \[ BAC = \left[\begin{array}{ccccc} 0 & & & &  (-1)^{n+1}b^n \\
                        1 & 0 & & &  \\
                          & 1 &0 & &  \\
                         &  & \ddots &\ddots & \\
                         & & &1& 0 \end{array} \right] \]
\end{lemma}
\begin{proof}
Use row and column operations and the  induction on $n \geq 2$ to prove this lemma.
\end{proof}







\begin{corollary}\label{L.20}
Let $n $  be a positive integer such that $ n \geq 3 $ ,  $ 1 \leq k \leq n-1 $ and let $m= n-k$. Suppose that $u$ and $ v$ are two variables on $\mathfrak{P}$ and let $A_1^{(k)} \in M_k(\mathfrak{P})$ and $A_2^{(k)} \in M_m(\mathfrak{P})$ be given by

$A_1^{(k)} = \begin{bmatrix} b & & & & &  \\
                        1 & b & & & & \\
                          & 1 & b & & & \\
                          & & \ddots & \ddots & & \\
                         & & &1& b   \end{bmatrix}  $ and $A_2^{(k)} = \begin{bmatrix} b & & & & &  \\
                        1 & b & & & & \\
                          & 1 & b & & & \\
                          & & \ddots & \ddots & & \\
                         & & &1& b   \end{bmatrix}  $  .

 If $B_k = \left[ \begin{array}{c|c} A_1^{(k)} & L_{1,m}^{k \times m}(v) \\ \hline  L_{1,k
 }^{m \times k}(u) & A_2^{(k)} \end{array}\right]  =
 \left[ \begin{array}{cccc|cccc} b & & & & & & &  v  \\
                                    1 & b & & & & & &   \\
                                      &  \ddots & \ddots &  & & & &  \\
                                      & & 1& b & & & &  \\ \hline
                                      & & &u &    b & & &  \\
                                      & & & &  1 & b & &  \\
                                      & & & &   &  \ddots & \ddots &  \\
                                      & & & &   & &1 & b
                                        \end{array}\right]  $,    then $ B_k$ is equivalent to the matrix $C_k= I_{n-2} \oplus \left[
                                                                                                                                  \begin{array}{cc}
                                                                                                                                    (-1)^{k+1}b^k & v \\
                                                                                                                                    u & (-1)^{m+1}b^m \\
                                                                                                                                  \end{array}
                                                                                                                                \right]\in M_n(\mathfrak{P})$ where $I_{n-2}$ is the identity matrix in $M_{n-2}(\mathfrak{P})$.

Moreover, if $D \in M_n(\mathfrak{P})$ is given by $ D = \begin{bmatrix} b & & &  &uv  \\
                        1 & b & & & \\
                          & 1 & b &  & \\
                          & & \ddots & \ddots &  \\
                        & & &1& b   \end{bmatrix}$ , then $D$ is equivalent to the matrix $\tilde{D}= I_{n-2}\oplus \left[
                                                                                                                      \begin{array}{cc}
                                                                                                                        (-1)^{n}b^{n-1} & uv \\
                                                                                                                        1 & b \\
                                                                                                                      \end{array}
                                                                                                                    \right]\in M_n(\mathfrak{P})$  where $I_{n-2}$ is the identity matrix in $M_{n-2}(\mathfrak{P})$.
\end{corollary}

\begin{proof}
By Lemma \ref{L.16}, there exist upper triangular matrices $B_1, C_1 \in M_k(\mathfrak{P})$ and $B_2, C_2 \in M_{n-k}(\mathfrak{P})$ with $1$ along their diagonal such that
$$ B_1 A_1^{(k)}C_1 = \left[ \begin{array}{cccc} 0 & & & (-1)^{k+1}b^k \\
                     1 & 0 & &  \\
                      & \ddots & \ddots &  \\
                      & &1 & 0 \end{array} \right] ,
 B_2 A_2^{(k)}C_2 = \left[ \begin{array}{cccc} 0 & & & (-1)^{m+1}b^m \\
                     1 & 0 & &  \\
                      & \ddots & \ddots &  \\
                      & &1 & 0 \end{array} \right] $$
where $m=n-k$.
Define $B,C \in M_n(\mathfrak{P})$ to be
$B = \begin{bmatrix} B_1 & 0 \\ 0 & B_2 \end{bmatrix} $ and $C = \begin{bmatrix} C_1 & 0 \\ 0 & C_2 \end{bmatrix} $.
Using   appropriate row and column operations  on the matrix $  B B_kC $ yields the required result. Applying a similar argument on the matrix $D$ proves the corresponding result.
\end{proof}

If $A$ and $B$ are matrices in $M_n(R)$ such that $\Cok_R(A)$ is isomorphic to  $\Cok_R(B)$, this does not imply  in general that $A$ is equivalent to $B$ \cite{LR}. However, if $R$ is a semiperfect ring (in particular, if $R$ is a commutative noetherian  local ring) and  $\Cok_R(A)$ is isomorphic to  $\Cok_R(B)$ as $R$-modules, then $A$ is equivalent to $B$ \cite[Theorem 4.3]{LR}.

Matrix factorizations were introduced by David Eisenbud in \cite{ED}  as a means of compactly describing
the minimal free resolutions of  maximal Cohen-Macaulay modules that have no free direct summands  over a local hypersurface
ring.
\begin{definition}\cite[Definition 1.2.1]{DI}
Let $f$ be a non-zero element of a commutative ring $S$.

 A matrix factorization of $f$ is a pair $(\phi,\psi)$ of homomorphisms
between finitely generated free $S$-modules  $ \phi: G \rightarrow F$  and $\psi: F \rightarrow G $, such
that
$\psi\phi = f I_G$ and $\phi\psi = f I_F$ .
\end{definition}

\begin{proposition}\label{C3.2}
Let $f$ be a non-zero element of a commutative ring $S$. If $( \phi: G \rightarrow F,\psi: F \rightarrow G )$ is a matrix factorization of $f$, then:
\begin{enumerate}
\item[(a)] $f\Cok( \phi )= f \Cok( \psi )=0$ .
\item[(b)] If $f$ is a  non-zerodivisor, then $\phi$ and $\psi$ are injective.
\item[(c)] If $S$ is a domain, then $G$ and $F$ are finitely generated free modules having the same rank.
\end{enumerate}
\end{proposition}
\begin{proof}
It is easy to show (a) and (b) but (c)can be proved using the same argument as in \cite[Page 127]{CMR}.

\end{proof}

As a result, we can define the matrix factorization of a non-zero element $f$ in a domain as the following.

\begin{definition}\label{D23}
Let $S$ be a domain and let $f\in S$ be a non-zero element. A matrix factorization is a pair $(\phi,\psi)$  of $n \times n$  matrices with
entries in $S$ such that $\psi\phi = \phi\psi = fI_n$ where $I_n$ is the identity matrix in $M_n(S)$. By $\Cok_S(\phi,\psi)$ and $\Cok_S(\psi,\phi)$, we mean $\Cok_S(\phi)$ and $\Cok_S(\psi)$ respectively. There are two distinguished trivial matrix factorizations of any element
$f$, namely $( f ,1)$ and $(1, f )$. Note that $\Cok_S(1, f ) = 0$, while $\Cok_S(f ,1)=S/fS$. Two matrix factorizations $(\phi,\psi)$ and $(\alpha,\beta)$  of $f$ are said to be equivalent and we write $(\phi,\psi)\sim(\alpha,\beta)$ if $\phi,\psi,\alpha,\beta \in M_n(S)$ for some positive integer $n$ and there exist invertible matrices $V,W \in M_n(S)$ such that $V\phi = \alpha W$ and $ W \psi = \beta V$. Furthermore, we define $(\phi,\psi)\oplus(\alpha,\beta)$ to be the matrix  $(\phi \oplus \alpha, \psi \oplus \beta)$ which is a matrix factorization of $f$. If $(S,\mathfrak{m})$ is a local domain,    a matrix factorization $(\phi,\psi)$ of an element   $f \in \mathfrak{m}\setminus\{0\}$ is reduced if all entries of $\phi$ and $\psi$ are in $\mathfrak{m}$.
\end{definition}

A matrix factorization can be  decomposed as follows.
 \begin{proposition}\label{P.24}\cite[Page58]{YY}
Let $(S,\mathfrak{m})$ be a regular local ring and  $f \in \mathfrak{m} \smallsetminus  \{0\}$.
Any  matrix factorization $(\phi,\psi)$ of $f$ can be written uniquely up to equivalence  as
\begin{equation*}\label{E4}
  (\phi,\psi) = (\alpha,\beta)\oplus (f,1)^t\oplus (1,f)^r
\end{equation*}
with $(\alpha,\beta)$ reduced, i.e, all entries of $\alpha,\beta$ are in $\mathfrak{m}$, and with $t, r$ non-negative integers.
\end{proposition}

A non-zero $R$-module $M$ is decomposable provided there exist non-zero R-modules $M_1, M_2$ such that $M = M_1 \oplus M_2$; otherwise $M$ is indecomposable.  If  $M$ is an $R$-module that does not have
a direct summand isomorphic to $R$, we say that $M$  is stable $R$-module.\\

D.  Eisenbud has established a relationship between  reduced matrix factorizations and stable  MCM modules as follows.
\begin{proposition}\label{A}\cite[ Corollary 7.6]{YY} \cite[Theorem 8.7]{CMR}
Let $(S,\mathfrak{m})$ be a regular local ring and let $f$ be a non-zero
element of $\mathfrak{m}$ and $R=S/fS$ . Then
the association $ (\phi, \psi)\mapsto \Cok(\phi, \psi)$   yields a bijective correspondence between the
set of equivalence classes of reduced matrix factorizations of f and the set of isomorphism
classes of stable MCM modules over R .
\end{proposition}
Based on Proposition \ref{A} and \cite[Corollary 3.7]{NW}, one can deduce the following remark.
\begin{remark}\label{C3.14}
Let $(S,\mathfrak{m})$ be a regular local ring, $f \in \mathfrak{m} \smallsetminus \{0\}$, and $R=S/fS$. If $(\phi, \psi)$  is a  reduced matrix factorization  of $f$, then $$(\phi, \psi)\sim [(\phi_1,\psi_1)\oplus (\phi_2,\psi_2)\oplus \dots\oplus(\phi_n,\psi_n)]$$ where $(\phi_i,\psi_i)$ is a reduced matrix factorization of $f$ and $\Cok_S(\phi_i,\psi_i)$ is non-free indecomposable MCM $R$-module for all $ 1 \leq i \leq n$. Furthermore,  the above representation of $(\phi, \psi)$  is unique up to equivalence when $S$ is complete.
\end{remark}
If $(S,\mathfrak{m})$  is a regular local ring, $f \in \mathfrak{m} \setminus \{0\}$ and    $u,v,z$ are variables,   a matrix factorization of $f+uv$  in $S[\![u,v]\!]$ and a matrix factorization of $f+z^2$ in $S[\![z]\!]$ can be obtained from a matrix factorization of $f$ as follows:
\begin{remark}\label{R3.8}
Let $(S,\mathfrak{m})$  be a regular local ring,   $f \in \mathfrak{m} \setminus \{0\}$, $R=S/fS$ and let $ u$ and $v$ be two variables.
Suppose that $(\phi,\psi)$ and $(\alpha,\beta)$ are two $n \times n$ matrix factorizations of $f$. Then
\begin{enumerate}

\item[(1)]  We define  $ (\phi, \psi)^{\maltese}$ to be $ ( \begin{bmatrix}  \phi & -vI \\  uI & \psi \end{bmatrix} , \begin{bmatrix}  \psi & vI \\  -uI & \phi \end{bmatrix})$ which  is a matrix factorization of  $f+uv$ in $S[\![u,v]\!]$ and we can see the following observations:
\begin{enumerate}

\item[(a)]  If $(\phi,\psi)\sim(\alpha,\beta)$, then  $(\phi,\psi)^{\maltese}\sim(\alpha,\beta)^{\maltese}$.
\item[(b)]  $[(\phi, \psi)\oplus(\alpha,\beta)]^{\maltese}$ is equivalent to $ (\phi, \psi)^{\maltese}\oplus(\alpha,\beta)^{\maltese}$.
\item[(c)]If $M$ is a stable MCM $R$-module, then $M=\Cok_S(\phi,\psi)$ where $(\phi,\psi)$ is a reduced matrix factorization of $f$ and hence we can  define $M^{\maltese}=\Cok_{S[\![u,v]\!]}(\phi,\psi)^{\maltese}$.
 As a result, if $M_i=\Cok_S(\phi_i,\psi_i)$ where
$(\phi_i,\psi_i)$ is a reduced  matrix factorization of $f$ for all $1 \leq i \leq n$, it follows that $( \bigoplus\limits_{j=1}^{n}M_j)^{\maltese} =\bigoplus \limits_{j=1}^{n}M_j^{\maltese}$.
\item[(d)]  If $R^{\bigstar}=S[\![u,v]\!]/(f+uv)$, then   $\Cok_{S[\![u,v]\!]}(f, 1)^{\maltese}=R^{\bigstar}= \Cok_{S[\![u,v]\!]}(1, f)^{\maltese}$  and hence we can write $(R)^{\maltese}= R^{\bigstar}$ as $R=\Cok_S(f,1)$.
\end{enumerate}

\end{enumerate}
\begin{enumerate}

\item[(2)]   We define $ (\phi, \psi)^{\sharp}$ to be $( \begin{bmatrix}  \phi & -zI \\  zI & \psi \end{bmatrix} ,
\begin{bmatrix}  \psi & zI \\  -zI & \phi \end{bmatrix})$ which is   a matrix factorization of $f+z^2$ in $S[\![z]\!]$ and we can see the following observations:
\begin{enumerate}
\item[(e)] If $(\phi,\psi)\sim(\alpha,\beta)$, then  $(\phi,\psi)^{\sharp}\sim(\alpha,\beta)^{\sharp}$.
\item[(f)] $ [(\phi, \psi)\oplus(\alpha,\beta)]^{\sharp}$ is equivalent to $ (\phi, \psi)^{\sharp}\oplus(\alpha,\beta)^{\sharp}$.
\item[(g)]  If $R^{\sharp}=S[\![z]\!]/(f+z^2)$, then $R^{\sharp}=S[\![z]\!]/(f+z^2)= \Cok_{S[\![z]\!]}(f, 1)^{\sharp} = \Cok_{S[\![z]\!]}(1,f)^{\sharp}$.
\end{enumerate}
\end{enumerate}
\end{remark}

 If $R$ is as  in the above Remark, the indecomposable non-free MCM modules over $R$ and $R^{\bigstar}$ can be  related in the following situation.

\begin{proposition}\label{P28}  \cite[Theorem 8.30]{CMR}
Let $(S, \mathfrak{m}, K)$ be a complete regular local ring such that $K$ is algebraically closed of characteristic not $2$ and $f\in \mathfrak{m}^2 \smallsetminus \{0\}$. If $R=S/fS$ and $R^{\bigstar}:=S[\![u,v]\!]/(f+uv)$, then the association $M\rightarrow M^{\maltese}$ defines a bijection between the isomorphisms classes of indecomposable non-free MCM modules over $R$ and $R^{\bigstar}$.
\end{proposition}
The following proposition is a direct consequence  of Proposition \ref{P28}, Remark \ref{R3.8}, and Remark \ref{C3.14}.
\begin{proposition}\label{P3.15}
Let $(S, \mathfrak{m}, K)$ be a complete regular local ring such that $K$ is algebraically closed of characteristic not $2$, and   $f\in \mathfrak{m}^2 \smallsetminus \{0\}$. let $R=S/fS$ and  $R^{\bigstar}:=S[\![u,v]\!]/(f+uv)$. If $(\phi, \psi)$ and $(\alpha, \beta)$ are reduced matrix factorizations of $f$, then $\Cok_S(\phi, \psi)$ is isomorphic to $\Cok_S(\alpha, \beta)$ over $R$ if and only if $\Cok_{S[\![u,v]\!]}(\phi, \psi)^{\maltese}$ is isomorphic to $\Cok_{S[\![u,v]\!]}(\alpha,\beta)^{\maltese}$  over $R^{\bigstar}.$
\end{proposition}

If $M$ is an $R$-module,  $ \sharp ( M,R)$ denotes the maximal rank of a free direct summands of $M$.
One can use Proposition \ref{A} and  Remark \ref{R3.8} to show the following corollary.
\begin{corollary}\label{C3.8}
Let $(S,\mathfrak{m})$ be a regular local ring ,  $f \in\mathfrak{ m} \smallsetminus \{0\}$,  $R=S/fS$, $R^{\bigstar}=S[\![u,v]\!]/(f+uv)$ , and $R^{\sharp}=S[\![z]\!]/(f+z^2)$ where $u,v$ and $z$ are  variables over $S$ .
If  $(\phi,\psi)$ is a matrix factorization of $f$ having the decomposition
\begin{equation*}\label{E4}
  (\phi,\psi) = (\alpha,\beta)\oplus (f,1)^t \oplus (1,f)^r
\end{equation*}
such that  $(\alpha,\beta)$ is  reduced and $t, r$ are non-negative integers, then :
\begin{enumerate}
\item[(a)] $\sharp (\Cok_S(\phi,\psi),R)=t$ and $\sharp (\Cok_S(\psi,\phi),R)=r$.
\item[(b)] $\sharp (\Cok_{S[\![u,v]\!]}(\phi,\psi)^{\maltese},R^{\bigstar})=\sharp (\Cok_S(\phi,\psi),R)+\sharp (\Cok_S(\psi,\phi),R)$.
\item[(c)] $\sharp (\Cok_{S[\![z]\!]}(\phi,\psi)^{\sharp},R^{\sharp})=\sharp (\Cok_S(\phi,\psi),R)+\sharp (\Cok_S(\psi,\phi),R)$.
\end{enumerate}
\end{corollary}





\section{The presentation of  $F_*^e(S/fS)$ as a cokernel of a Matrix Factorization of $f$}
\label{Section: The presentation as a cokernel of a Matrix Factorization}

Throughout the rest of this paper, unless otherwise mentioned, we will adopt the following notation:
\begin{notation}\label{N4.1}
$K$ will denote a  field  of prime characteristic $p$ with $[K:K^p] < \infty $, and we set $q=p^e$ for some $e \geq 1$. $S$ will denote the ring  $K[x_1,\dots,x_n]$ or ($K[\![x_1,\dots,x_n]\!]$). Let $\Lambda_e$ be  a basis of $K$ as $K^{p^e}$-vector space.   We set $$ \Delta_e := \{\lambda x_1^{a_1}\dots x_n^{a_n}\,|\, 0 \leq a_i \leq p^e-1 \text{ for all } 1\leq i \leq n \text{ and } \lambda \in \Lambda_e \} $$
and set  $r_e := | \Delta_e| = [K:K^p]^eq^n. $
\end{notation}
\begin{discussion}\label{D4.2}
It is straightforward to see that $ \{ F_*^e(j)\, |\, j \in \Delta_e \} $ is a basis of $F_*^e(S)$ as free $S$-module. Let  $f \in S$. If $S \xrightarrow{f} S $ is the $S$-linear map given by $s \longmapsto fs$, let $F_*^e(S) \xrightarrow{F_*^e(f)} F_*^e(S) $ be the $S$-linear map that is given by $F_*^e(s) \longmapsto F_*^e(fs)$ for all $s\in S$. We write $M_S(f,e)$ (or $M(f,e)$ if $S$ is known) to denote the  $r_e \times r_e$  matrix representing the $S$-linear map $F_*^e(S) \xrightarrow{F_*^e(f)} F_*^e(S) $ with respect to the basis $ \{ F_*^e(j)\, |\, j \in \Delta_e \} $. Indeed, if $j \in \Delta_e $, there exists a unique  set  $\{f_{(i,j)} \in S \,|\, i \in \Delta_e \} $ such that $F_*^e(jf)= \bigoplus_{i\in\Delta_e}f_{(i,j)}F^e(i)$ and consequently $M_S(f,e)=[f_{(i,j)}]_{(i,j) \in \Delta^2_e}$.   The   matrix $M_S(f,e)$   is called the matrix of relations of $f$ over $S$ with respect to $e$.
\end{discussion}
\begin{example}
Let $K$ be a perfect field of prime characteristic 3 , $S=K[x,y]$ or $S=k[\![x,y]\!]$ and  let $ f =x^2 + xy$.  We aim to construct  $M_S(f,1)$. Let    $\{ F_*^1(1) , F_*^1(x) , F_*^1(x^2) , F_*^1(y) , F_*^1(yx)
, F_*^1(yx^2),F_*^1(y^2) , \\ F_*^1(y^2x) , F_*^1(y^2x^2) \}$ be the ordered basis of $F_*^1(S)$ as  $S$-module. Therefore, $M_S(f,1)$ is the $9 \times 9$ matrix whose $j$-th column consists  of the coordinates of  $F_*^e(jf)$  with respect to the given basis where $F_*^e(j)$ is the $j$-th element of the above basis.  For example, the $8$-th column of $M_S(f,1)$ is obtained from $ F_*^1(y^2xf) = F_*^1(y^2x^3 + x^2y^3) = xF_*^1(y^2) + yF_*^1(x^2) $. \\
As a result, it follows that
$$ M_S(f,1)= \begin{bmatrix}  0 & x & 0 & 0 & 0 & 0 & 0 & 0 & yx \\ 0 & 0 & x & 0 & 0 & 0 & y & 0 & 0 \\ 1 & 0 & 0 & 0 & 0 & 0 & 0& y & 0 \\ 0 & 0 & x & 0 & x & 0 & 0 & 0 & 0 \\ 1 & 0 & 0 & 0 & 0 & x & 0 & 0 & 0 \\ 0 & 1 & 0 & 1 & 0 & 0 & 0 & 0 & 0 \\ 0 & 0 & 0 & 0 & 0 & x & 0 & x & 0 \\ 0 & 0 & 0 & 1 & 0 & 0 & 0 & 0 & x \\ 0 & 0 & 0 & 0 & 1 & 0 & 1 & 0 & 0 \end{bmatrix}. $$
\end{example}

\begin{remark}\label{Ex4.3}
If $m \in \mathbb{N}$, then $F_*^e(f^{mq}j)=f^mF_*^e(j)$ for all $j \in \Delta_e$. This makes $M_S(f^{mq},e)=f^mI$ where $I$ is the identity matrix of size $r_e \times r_e$.
\end{remark}

\begin{proposition}\label{P4.3}

If $f,g  \in S$, then
\begin{enumerate}
\item[(a)] $ M_S(f+g,e)=M_S(f,e)+M_S(g,e)$,
\item[(b)] $ M_S(fg,e)=M_S(g,e)M_S(f,e)$ and consequently $ M_S(f,e)M_S(g,e)=M_S(g,e)M_S(f,e)$, and
\item[(c)] $  M_S(f^m,e)= [M_S(f,e)]^m$ for all $m \geq 1$.
\end{enumerate}
\end{proposition}
\begin{proof}
The proof follows immediately from Discussion \ref{D4.2}.



\end{proof}

According to  Remark \ref{Ex4.3} and Proposition \ref{P4.3}, we get that $$M_S(f^k,e)M_S(f^{q-k},e)=M_S(f^{q-k},e)M_S(f^k,e)= M_S(f^q,e)= fI_{q^n}$$ for all $0 \leq k \leq q$. This shows the following result.

\begin{proposition}\label{P4.5}
For every $f \in S$ and $0 \leq k \leq q$, the pair  $ ( M_S(f^k,e),M_S(f^{q-k},e))$ is a matrix factorization of $f$.

\end{proposition}
\begin{discussion}

Let  $x_{n+1}$ be a new variable  and let $ L=S[x_{n+1}]$  if $S=K[x_1,\ldots,x_n]$ or ($ L=S[\![x_{n+1}]\!]$  if $S=K[\![x_1,\ldots,x_n]\!]$). We aim to describe $M_L(g,e)$ for some $g \in L$ by describing the columns of $M_L(g,e)$. First we will construct a basis  of the free $L$-module  $F_*^e(L)$   using  the basis $ \{ F_*^e(j) | j \in\Delta_e \} $  of the free $S$-module $F_*^e(S)$.  For each $0\leq v \leq q-1 $, let $ \mathfrak{B}_v= \{F_*^{e}(j x_{n+1}^{v})\, |\, j \in\Delta_e  \}$  and set  $ \mathfrak{B} = \mathfrak{B}_0 \cup \mathfrak{B}_1 \cup \mathfrak{B}_2\cup \dots \cup \mathfrak{B}_{q-1} $.
Therefore $\mathfrak{B}$  is a basis for $F_*^e(L)$ as free $L$-module and if   $g \in L$,  we write
 \begin{eqnarray*}F_*^e(g)& = & \bigoplus\limits_{i\in\Delta_e}g_{i}^{(0)}F^{e}(i) \oplus\bigoplus\limits_{i\in\Delta_e}g_{i}^{(1)}F^{e}(ix_{n+1}^1)\oplus\dots \oplus \bigoplus\limits_{i\in\Delta_e}g_{i}^{(q-1)}F^{e}(ix_{n+1}^{q-1})
\end{eqnarray*}
where $\{ g_{i}^{(s)} \in L\, |\,  0 \leq s \leq q-1 \text{ and } i\in\Delta_e \}$.
For each $ 0 \leq s \leq q-1 $ let   $ [F_*^e(g)]_{\mathfrak{B}_s} $ denote the column whose entries are the coordinates $ \{g_{i}^{(s)}\,|\, i\in\Delta_e\} $ of $F_*^e(g)$ with respect to $\mathfrak{B}_s$. Let $[F_*^e(g)]_\mathfrak{B}$ be the $r_eq \times 1$  column that is composed of the columns $[F_*^e(g)]_{\mathfrak{B}_0},\dots,[F_*^e(g)]_{\mathfrak{B}_{q-1}}$ respectively.
Therefore  $M_L(g,e)$ is the  $r_eq \times r_eq$ matrix over $L$ whose columns are all the columns  $ [F_*^e(jx_{n+1}^s g)]_\mathfrak{B}$ where $ 0 \leq s \leq q-1$ and $ j \in \Delta_e$. This means that
 $ M_L(g,e) = \left[\begin{array}{ccc}
                 C_{0} & \ldots & C_{q-1} \\
               \end{array}
             \right]$  where $C_{m}$ is the  $r_eq \times r_e$ matrix over $L$ whose columns are the columns $ [F_*^e(jx_{n+1}^m g)]_\mathfrak{B}$  for all $j \in \Delta_e$. If we define $C_{(k,m)}$ to be the  $r_e \times r_e$ matrix over $L$ whose columns are $ [F_*^e(jx_{n+1}^m g)]_{\mathfrak{B}_k}$ for all $j\in\Delta_e$, then $C_{m}$ consists of $ C_{(0,m)},\dots,C_{(q-1,m)} $ respectively
and hence the matrix
 $M_L(g,e)$ is given by :
\begin{equation}
  M_L(g,e) = \left[
               \begin{array}{ccc}
                 C_{0} & \ldots & C_{q-1} \\
               \end{array}
             \right]=\left[
               \begin{array}{ccc}
                 C_{(0,0)} & \ldots & C_{(0,q-1)} \\
                 \vdots&   & \vdots \\
                 C_{(q-1,0)} & \ldots & C_{(q-1,q-1)} \\
               \end{array}
             \right].
\end{equation}

\end{discussion}
Using the above discussion we can prove the following lemma
\begin{lemma}\label{L4.7}
Let $f \in S$ with $A = M_S(f,e)$ and let $ L=S[x_{n+1}]$  if $S=K[x_1,\ldots,x_n]$  or ($ L=S[\![x_{n+1}]\!]$  if $S=K[\![x_1,\ldots x_n]\!]$). If $0\leq d \leq q-1 $, then
\begin{equation}
  M_L(fx_{n+1}^d,e) = \left[
               \begin{array}{ccc}
                 C_{(0,0)} & \ldots & C_{(0,q-1)} \\
                 \vdots&   & \vdots \\
                 C_{(q-1,0)} & \ldots & C_{(q-1,q-1)} \\
               \end{array}
             \right]
\end{equation}
where
\begin{equation*}
    C_{(k,m)} = \begin{cases}
              A             &\text{if   } (k,m) \in \{ (d,0),(d+1,1),\ldots ,(q-1,q-1-d)\} \\
             x_{n+1}A                &\text{if   } (k,m) \in \{ (0,q-d),(1,q-1-d),\ldots ,(d,q-1)\}\\
              0               & \text{otherwise}
           \end{cases}
\end{equation*}

\end{lemma}

\begin{proof}
If $A = M_S(f,e)=[f_{(i,j)}]$, for each $j \in \Delta_e$ we can write  $F_*^e(jf)= \bigoplus_{i\in \Delta_e}f_{(i,j)}F^e(i)$. If $g=fx_{n+1}^d$, for every $1 \leq m \leq q-1 $ and   $j \in \Delta_e$, it follows that $F_*^e(jx_{n+1}^mg) = \bigoplus_{i\in \Delta_e}f_{(i,j)}F^e(ix_{n+1}^{d+m})$. Therefore,

\begin{equation*}
    F_*^e(jx_{n+1}^mg) = \begin{cases}
             \bigoplus_{i\in \Delta_e}f_{(i,j)}F^e(ix_{n+1}^{d+m})   &\text{if   } d+m \leq q-1 \\
           \bigoplus_{i\in \Delta_e}x_{n+1}f_{(i,j)}F^e(ix_{n+1}^{d+m -q })             &\text{if   } d+m > q-1
           \end{cases}
\end{equation*}
 Accordingly, if $m \leq q-1-d $, then



 \begin{equation*}
    C_{(k,m)} = \begin{cases}
             A             &\text{ if } k=d+m \\
              0            &\text{ if } k \neq d+m
           \end{cases}
\end{equation*}
  However,  if $m > q-1-d $, it follows that


 \begin{equation*}
    C_{(k,m)} = \begin{cases}
              x_{n+1}A              &\text{if } k=d+m -q \\
              0                     &\text{if } k \neq d+m -q
           \end{cases}
\end{equation*}

  This shows the required result.

\end{proof}

\begin{proposition}\label{EP}
Let $ L=S[x_{n+1}]$ (if $S=K[x_1,\ldots,x_n]$)or $ L=S[\![x_{n+1}]\!]$ (if $S=K[\![x_1,\ldots,x_n]\!]$) . Suppose that $g \in L$ is given by $$ g = g_0 + g_1x_{n+1} + g_2x_{n+1}^2 + \dots +g_dx_{n+1}^d$$ where $d < q$ and $g_k \in S$ for all $ 0 \leq k \leq d $ . If $A_k=M_S(g_k,e)$ for each $0 \leq k \leq d $ then  \[ M_L(g,e) = \begin{bmatrix}
A_0& &  & & x_{n+1}A_d&x_{n+1}A_{d-1} &\ldots &x_{n+1}A_1  \\
A_1&A_0&  & & &x_{n+1}A_d &   &\vdots \\
 & & & & & &\ddots & \\
 & & & & & & &x_{n+1}A_d \\
\vdots & \vdots & \ddots & \ddots& & & & \\
A_d&\vdots&  & & & & & \\
  & A_d & & & & & &\\
  &  &  & & & & & \\
  &  &  &A_d &A_{d-1} &A_{d-2}& \ldots &A_0

\end{bmatrix}\]

\end{proposition}
\begin{proof}
By Proposition [\ref{P4.3}], we get
$$ M_L(g,e) = M_L(g_0,e) + M_L(g_1x_{n+1},e) + M_L(g_2x_{n+1}^2,e) + \dots +M_L(g_dx_{n+1}^d,e)$$

Applying Lemma \ref{L4.7} for $M_L(g_jx_{n+1}^j,e)$ for all $0\leq j \leq n$ yields the result.




\end{proof}

\begin{example}
Let $K$ be a perfect field of prime characteristic 3 and let $S=K[x]$ or $S=K[\![x]\!]$ . Assume $ L=S[y]$ (if $S=K[x]$) or $L=S[\![y]\!]$ (if $S=K[\![x]\!]$ ).  Let $ f =x^2 + xy$, $f_0=x^2$, and $f_1=x$.

By Proposition \ref{EP} it follows that

$$ M_L(f,1)=\begin{bmatrix} M_S(f_0,1) & & yM_S(f_1,1) \\ M_S(f_1,1) & M_S(f_0,1) &  \\ & M_S(f_1,1) & M_S(f_0,1)  \end{bmatrix} = \left[\begin{array}{ccc|ccc|ccc}  0 & x & 0 & 0 & 0 & 0 & 0 & 0 & yx \\ 0 & 0 & x & 0 & 0 & 0 & y & 0 & 0 \\ 1 & 0 & 0 & 0 & 0 & 0 & 0& y & 0 \\ \hline 0 & 0 & x & 0 & x & 0 & 0 & 0 & 0 \\ 1 & 0 & 0 & 0 & 0 & x & 0 & 0 & 0 \\ 0 & 1 & 0 & 1 & 0 & 0 & 0 & 0 & 0 \\ \hline 0 & 0 & 0 & 0 & 0 & x & 0 & x & 0 \\ 0 & 0 & 0 & 1 & 0 & 0 & 0 & 0 & x \\ 0 & 0 & 0 & 0 & 1 & 0 & 1 & 0 & 0 \end{array} \right] $$
\end{example}



\begin{proposition}\label{P4.11}
Let $f \in S$ be a non-zero non-unit element. If $R=S/fS$, then
$F_*^e(R)$ is a maximal Cohen-Macaulay $R$-module isomorphic to $\Cok_S(M_S(f,e))$ as $S$-modules (and as $R$-modules).
\end{proposition}

\begin{proof}

The first assertion follows from the fact that $R$ is Cohen Macaulay.
Write $I=f S$.  Since $ \{ F_*^e(j) \, | \, j \in \Delta_e \} $ is a basis of $F_*^e(S)$ as free $S$-module,
the module $F_*^e(R)$ is generated as $S$-module by the set $ \{ F_*^e(j + I) \,| \,j \in \Delta_e \} $. For every $g \in S$, define
$ \phi( F_*^e(g)) = F_*^e(g + I)$. It is clear that $\phi: F_*^e(S)  \longrightarrow F_*^e(R)$ is a surjective homomorphism of $S$-modules whose kernel is the $S$-module
$F_*^e(I)$ that is generated by the set $ \{ F_*^e(jf) \, | \, j \in \Delta_e \} $.
Now, define the $S$-linear map  $ \psi : F_*^e(S) \rightarrow F_*^e(S)$  by $ \psi( F_*^e(h)) = F_*^e(hf)$ for all $h\in S$.
We have an exact sequence $F_*^e(S)\xrightarrow{\psi}F_*^e(S)\xrightarrow{\phi}F_*^e(R) \xrightarrow{} 0$. Notice for each $j \in \Delta_e$  that $\psi(F_*^e(j))= F_*^e(jf)= \bigoplus_{i\in\Delta_e}f_{(i,j)}F^e(i)$
and hence $M_S(f,e)$
represents the map $\psi$ on the given free-bases.
By  Proposition \ref{P4.5} and  Proposition \ref{C3.2}(a), it follows  that  $F_*^e(R)$ is isomorphic to $\Cok_S(M_S(f,e))$ as $R$ -modules.\\

\end{proof}
\begin{corollary}\label{C4.12}
Let $f\in S$ be a non-zero non-unit element. If $1 \leq k \leq q-1$ and $R=S/fS$, then
\begin{enumerate}
\item[(a)] $F_*^e(S/f^kS)$ is a maximal Cohen-Macaulay $R$-module isomorphic to $Cok_S(M_S(f^k,e))$ as $S$-modules (and as $R$-modules), and
\item[(b)]  $F_*^e(S/f^kS)$ is a maximal Cohen-Macaulay $S/f^kS$ -modules isomorphic to  $Cok_S(M_S(f^k,e))$ as $S$-modules (and as $S/f^kS$-modules).
\end{enumerate}
\end{corollary}
\begin{proof}
(a) Since  the pair $(M_S(f^k,e),M_S(f^{q-k},e))$ is a matrix factorization of $f$ (Proposition  \ref{P4.5}), it follows that $f Cok_S(M_S(f^k,e))=0$. It is clear that $fF_*^e(S/f^kS)=0$ .
This makes $F_*^e(S/f^kS)$ and  $Cok_S(M_S(f^k,e))$ $R$-modules and consequently  $F_*^e(S/f^kS)$ is isomorphic to $Cok_S(M_S(f^k,e))$ as $R$ -modules. \\
The result (b) can be proved by  observing  that  $(M_S(f^k,e),M_S(f^{kq-k},e))$ is a matrix factorization of $f^k$ and applying the same argument as above.
\end{proof}

\begin{lemma}\label{L4.13}
Let $K$ be a  field of prime characteristic $p>2 $ with $[K:K^p] <\infty$  and let $T=S[z]$ if $S=K[x_1,\dots,x_n]$ ( or $T=S[\![z]\!]$ if $S=K[\![x_1,\dots,x_n]\!]$ ). If $A=M_S(f,e)$ for some $f \in S$,   then
\begin{equation*}
 F_*^e(T/(f+z^2))= \Cok_{T}\left[
                                              \begin{array}{cc}
                                                A^{\frac{q-1}{2}} & -zI \\
                                                zI &  A^{\frac{q+1}{2}} \\
                                              \end{array}
                                            \right]
\end{equation*}
\end{lemma}

\begin{proof}
Let $I$ be the identity matrix in $M_{r_e}(S)$  where $r_e=[K:K^p]^ep^{en}$. It follows by Proposition \ref{EP} that $ M_{T}(f+z^2, e)$ is a $q \times q$ matrix over the ring $M_{r_e}(S)$ that is given by

\begin{equation*}
  M_{T}(f+z^2, e) = \left[
       \begin{array}{ccccccc}
         A &   &  &   &   & zI & 0  \\
         0 & A &  &   &   &   & zI \\
         I & 0 & A &   &   &   &  \\
           & I & 0  & A &   &   &   \\
           &   & \ddots & \ddots & \ddots &   &  \\
           &   &   &  I & 0 & A &   \\
           &   &   &   &  I & 0 & A \\
       \end{array}
     \right]
\end{equation*}
By Corollary \ref{C2.14}, we get $ F_*^e(T/(f+z^2)) = \Cok_{T}\left[
                                              \begin{array}{cc}
                                                A^{\frac{q-1}{2}} & -zI \\
                                                zI &  A^{\frac{q+1}{2}} \\
                                              \end{array}
                                            \right] $
\end{proof}




\begin{proposition}\label{P21}
Let $u$ and $v$ be new variables on $S$ and let $L=S[u,v]$ if $S=K[x_1,\dots,x_n]$ ( or $L=S[\![u,v]\!]$  if $S=K[\![x_1,\dots,x_n]\!]$ ). Let  $R^{\bigstar}=L/(f+uv)$. If $A$ is the matrix $M_S(f,e)$ for some $f\in S$ and $I$ is the identity matrix in the ring $M_{r_e}(S)$, where $r_e=[K:K^p]^ep^{en}$,  then $$ F_*^e(L/(f+uv))= \Cok_F(M_L(f+uv,e))= (R^{\bigstar})^{r_e} \bigoplus \bigoplus_{j=1}^{q-1}\Cok_LB_j $$ where  $ B_k = \begin{bmatrix}  A^k & -vI \\ uI & A^{q-k}  \end{bmatrix}$ for all $ 1\leq k \leq q-1$.
\end{proposition}

\begin{proof}
Recall  that $\mathfrak{D}=\{ F_*^e(ju^sv^t) \,|\, j\in \Delta_e , 0 \leq s,t \leq q-1 \}$ is a free basis of $F_*^e(L)$ as $L$-module.
We introduce a $\mathbb{Z}/q\mathbb{Z}$-grading on both $L$ and $F_*^e(L)$ as follows:
$L$ is concentrated in degree 0, while $\deg(F_*^e(x_i))=0$ for each $1 \leq i \leq n$, $\deg(F_*^e(u))=1$ and $\deg(F_*^e(v))=-1$.
We can now write
$F_*^e(L)= \bigoplus_{k=0}^{q-1}M_k$ where $M_k$ is the free $L$-submodule of $F_*^e(L)$
of elements of homogeneous degree $k$, and which is generated by
$$\mathfrak{D}_k = \{F_*^e(ju^sv^t)\in \mathfrak{D}\,|\, \deg(F_*^e(ju^sv^t))=k \}.$$
Note that
$\mathfrak{D}_0= \{F_*^e(ju^sv^s)\,|\, j\in\Delta_e , 0 \leq s \leq q-1 \} $,  and that for all $1 \leq k \leq q-1$
\begin{eqnarray*}
  \mathfrak{D}_k&=& \{F_*^e(ju^{k+r}v^{r})\,|\, j\in\Delta_e , 0 \leq r \leq q-k-1 \} \cup\\
    & &  \{F_*^e(ju^{r}v^{q-k+r})\,|\, j\in\Delta_e ,  0 \leq r \leq k-1 \}.
\end{eqnarray*}
Let $J$ be the ideal $(f+uv)L$. Since $\deg(F_*^e(f+uv))=0$, it follows that $F_*^e(J)= \bigoplus_{k=0}^{q-1}M_kF_*^e(f+uv)$ and consequently
\begin{equation}\label{EEEE1}
  F_*^e(L/(f+uv))= F_*^e(L)/F_*^e(J)=\bigoplus_{k=0}^{q-1}M_k/M_kF_*^e(f+uv).
\end{equation}

We now show that
$M_k/M_kF_*^e(f+uv) \cong \Cok_LC_k  $  where $C_0= \begin{bmatrix} (-1)^{q}A^{q-1} & uvI \\ I & A \end{bmatrix} $ and $ C_k = \begin{bmatrix} (-1)^{q-k+1}A^{q-k} & vI \\ uI & (-1)^{k+1}A^{k} \end{bmatrix}$ for all $1\leq k \leq q-1$.
Recall that if $M_s(f,e)=[f_{(i,j)}]$, then $F_*^e(j)=\bigoplus_{i\in\Delta_e}f_{(i,j)}F_*^e(i)$ for all $j \in \Delta_e$.
So now
 \begin{eqnarray}\label{Eq18}
  \nonumber  F_*^e(ju^sv^t(f+uv)) &=& F_*^e(jf)F_*^e(u^sv^t) + F_*^e(ju^{s+1}v^{t+1}) \\
     &=& \bigoplus_{i\in\Delta_e}f_{(i,j)}F_*^e(iu^sv^t) \oplus F_*^e(ju^{s+1}v^{t+1}).
 \end{eqnarray}

Since $\deg(F_*^e(iu^sv^t))=\deg(F_*^e(ju^{s+1}v^{t+1}))$ for all $i,j  \in \Delta_e$ and all $0 \leq s,t \leq q-1$, it follows that $ F_*^e(ju^sv^t(f+uv)) \in M_k$ for all $F_*^e(ju^sv^t) \in \mathfrak{D}_k$. This enables us to define the homomorphism  $\psi_k :M_k\rightarrow M_k$ that is  given by $\psi_k(F_*^e(ju^sv^t))=F_*^e(ju^sv^t(f+uv))$ for all $F_*^e(ju^sv^t) \in \mathfrak{D}_k$ and consequently we have the following short exact sequence   $$M_k\xrightarrow{\psi_k}M_k\xrightarrow{\phi_k}M_k/M_kF_*^e(f+uv) \xrightarrow{} 0$$
where $\phi_k :M_k\rightarrow M_k/M_kF_*^e(f+uv)$ is the canonical surjection.
Notice that if $0 \leq  s < q-1$, equation (\ref{Eq18}) implies that
\begin{equation}\label{Eq19}
 F_*^e(ju^sv^s(f+uv))=  \bigoplus_{i\in\Delta_e}f_{(i,j)}F_*^e(iu^sv^s) \oplus F^e(ju^{s+1}v^{s+1})
\end{equation} and

 \begin{equation}\label{Eq20}
 F_*^e( ju^{q-1}v^{q-1}(f+uv))=  \bigoplus_{i\in\Delta_e}f_{(i,j)}F_*^e(iu^{q-1}v^{q-1}) \oplus uv F_*^e(j) ,
\end{equation}
therefore $\psi_0$ is represented by the matrix $\begin{bmatrix} A & & &uvI \\ I & A & & \\ & \ddots & \ddots & \\ & & I & A\end{bmatrix}$ which is a $q\times q$ matrix over the ring $M_{r_e}(L).$
Now Corollary \ref{L.20} implies that
\begin{equation}\label{EEEE2}
 M_0/M_0F_*^e(f+uv)\cong \Cok_L \begin{bmatrix} (-1)^{q}A^{q-1} & uvI \\ I & A \end{bmatrix}=(R^{\bigstar})^{r_e}
\end{equation}
Now let $1 \leq k \leq q-1$. If $0 \leq r < q-k-1$, then  it follows from equation (\ref{Eq18}) that
\begin{equation}\label{Eq21}
 F_*^e( ju^{k+r}v^r(f+uv))=  \bigoplus_{i\in\Delta_e}f_{(i,j)}F_*^e(iu^{k+r}v^{r}) \oplus F_*^e(ju^{k+r+1}v^{r+1})
\end{equation}
and
\begin{equation}\label{Eq22}
 F_*^e( ju^{q-1}v^{q-k-1}(f+uv))= \bigoplus_{i\in\Delta_e}f_{(i,j)}F_*^e(iu^{q-1}v^{q-k-1}) \oplus u F_*^e(jv^{q-k}).
\end{equation}
However, if $0 \leq r < k-1$, it follows from (\ref{Eq18}) that
\begin{equation}\label{Eq22}
 F_*^e( ju^{r}v^{q-k-r}(f+uv))= \bigoplus_{i\in\Delta_e}f_{(i,j)}F_*^e(iu^{r}v^{q-k-r}) \oplus  F_*^e(ju^{r+1}v^{q-k-r+1})
\end{equation}
and
\begin{equation}\label{Eq23}
 F_*^e( ju^{k-1}v^{q-1}(f+uv))=\bigoplus_{i\in\Delta_e}f_{(i,j)}F_*^e(iu^{k-1}v^{q-1}) \oplus vF_*^e(ju^{k}).
\end{equation}

As a result, $\psi_k$  is represented by the matrix  $ \left[ \begin{array}{cccc|cccc} A & & & & & & &  vI  \\
                                    I & A & & & & & &   \\
                                      &  \ddots & \ddots &  & & & &  \\
                                      & & I& A & & & &  \\ \hline
                                      & & &uI &    A & & &  \\
                                      & & & &  I & A & &  \\
                                      & & & &   &  \ddots & \ddots &  \\
                                      & & & &   & &I & A
                                        \end{array}\right] $ which is a $q\times q$ matrix over the ring $M_{r_e}(L)$  where   $uI$ is in the

                                         $(q-k+1,q-k)$ spot of this matrix.

Therefore, Corollary \ref{L.20} implies that
\begin{eqnarray}
 \nonumber M_k/M_kF_*^e(f+uv) &\cong & \Cok_L\begin{bmatrix} (-1)^{q-k+1}A^{q-k} & vI \\ uI & (-1)^{k+1} A^{k} \end{bmatrix}.
\end{eqnarray}
If $k$ is odd integer, we notice that

\begin{eqnarray}\label{EEq}
 \nonumber  M_k/M_kF_*^e(f+uv)  &\cong & \Cok_L( \left[
                                                                                              \begin{array}{cc}
                                                                                                -I & 0 \\
                                                                                                0 & I \\
                                                                                              \end{array}
                                                                                            \right] \begin{bmatrix} (-1)^{q-k+1}A^{q-k} & vI \\ uI & (-1)^{k+1} A^{k} \end{bmatrix}) \\
   &\cong & \Cok_L \begin{bmatrix} A^{q-k} & -vI \\ uI &  A^{k} \end{bmatrix}.
\end{eqnarray}

Similar argument when  $k$ is even shows that
\begin{equation}\label{Eq24}
  M_k/M_kF_*^e(f+uv)  \cong \Cok_L \begin{bmatrix} A^{q-k} & -vI \\ uI &  A^{k} \end{bmatrix}.
\end{equation}

Now the proposition follows from (\ref{EEEE1}), (\ref{EEEE2}), (\ref{EEq}) and (\ref{Eq24}).
\end{proof}

We shall need the following corollary later
\begin{corollary}\label{C4.19}
Let  $S=K[\![x_1,\dots,x_n]\!]$ and let $\mathfrak{m}$ be the maximal ideal of $S$.  Let  $R=S/fS$ where $f \in \mathfrak{m} \setminus \{0\}$. Let  $R^{\bigstar}= S[\![u,v]\!]/(f+uv)$. If $A=M_S(f,e)$, then
\begin{equation*}
  \sharp ( F_*^e(R^{\bigstar}),R^{\bigstar})= r_e + 2 \sum_{k=1}^{q-1} \sharp ( \Cok_S(A^k),R).
 \end{equation*}

\end{corollary}

\begin{proof}

By Proposition \ref{P21}, it follows that
\begin{equation*}
  F_*^e(R^{\bigstar})=  (R^{\bigstar})^{r_e} \bigoplus \bigoplus_{j=1}^{q-1} \Cok_{S[\![u,v]\!]} \begin{bmatrix}  A^k & -vI \\ uI & A^{q-k}  \end{bmatrix}.
\end{equation*}

However, For each $1 \leq k \leq q-1 $ , we  notice that

\begin{equation*}
  \Cok_{S[\![u,v]\!]}\begin{bmatrix}  A^k & -vI \\ uI & A^{q-k}  \end{bmatrix}=\Cok_{S[\![u,v]\!]}(A^k,A^{q-k})^{\maltese}.
\end{equation*}

Therefore, by Corollary \ref{C3.8} and the convention  that $\Cok_{S}(A^k,A^{q-k})=\Cok_{S}(A^k)$ it follows that

\begin{eqnarray*}
   \sharp ( F_*^e(R^{\bigstar}),R^{\bigstar}) &=& r_e +  \sum_{k=1}^{q-1}\sharp (\Cok_{S[[u,v]]}(A^k,A^{q-k})^{\maltese}, R^{\bigstar}) \\
    &=&  r_e +  \sum_{k=1}^{q-1} \left[\sharp(\Cok_{S}(A^k,A^{q-k}),R)+ \sharp(\Cok_{S}(A^{q-k},A^k), R)\right] \\
    &=&  r_e + 2 \sum_{k=1}^{q-1} \sharp (\Cok_{S}(A^k,A^{q-k}),R)\\
    &=&  r_e + 2 \sum_{k=1}^{q-1} \sharp (\Cok_{S}(A^k),R).
\end{eqnarray*}
\end{proof}


\section{When does $S[\![u,v]\!]/(f+uv)$ have finite F-representation type?}
\label{Section: When does f+uv have finite F-representation type?}

We  keep the same notation as in section \ref{Section: The presentation as a cokernel of a Matrix Factorization} unless otherwise stated. The purpose of this section is to provide a characterization of when the ring $S[\![u,v]\!]/(f+uv)$ has finite F-representation type. This characterization enables us of exhibiting a class of rings in section \ref{Section: A Class of rings that have FFRT but not finite CM type}  that have FFRT but not finite CM type.

\begin{proposition}\label{P30}
Let $K$ be an algebraically closed field of prime characteristic $p > 2$ and $q=p^e$. Let  $S:= K[\![x_1,\dots,x_{n}]\!]$ and let $ \mathfrak{m}$ be the maximal ideal of $S$ and $f \in \mathfrak{m}^2 \setminus \{0\}$. Let $R=S/(f)$ and $R^{\bigstar}=S[\![u,v]\!]/(f+uv)$. Then $R^{\bigstar}=S[\![u,v]\!]/(f+uv)$ has FFRT over $R^{\bigstar}$ if and only if there exist indecomposable  $R$-modules $N_1,\dots,N_t$ such that $F_*^e(S/(f^k))$  is  a direct sum with direct summands taken from $ N_1,\dots,N_t $ for every $e\in \mathbb{N}$ and  $1 \leq k < p^e $ .
\end{proposition}
\begin{proof}
First, suppose that $R^{\bigstar}=S[\![u,v]\!]/(f+uv)$ has FFRT over $R^{\bigstar}$ by $\{R^{\bigstar} , M_1, \dots , M_t , \}$ where $M_j$ is an indecomposable non-free MCM $R^{\bigstar}$-module
  By Proposition \ref{P28} and \ref{A},  it follows  that $M_j= \Cok_{S[\![u,v]\!]}(\alpha_j,\beta_j)^{\maltese}$ where $(\alpha_j,\beta_j)$ is a reduced matrix factorization of $f$ such that $\Cok_S(\alpha_j,\beta_j)$ is non-free indecomposable MCM $R$-module. Now apply Proposition \ref{P21} to get that
 \begin{equation}\label{E13}
  F_*^e(S[\![u,v]\!]/(f+uv)) \cong( R^{\bigstar})^{r_e} \oplus \bigoplus \limits_{k=1}^{q-1}\Cok_{S[\![u,v]\!]}(A^k,A^{q-k})^{\maltese }
 \end{equation}
 where $A=M_S(f,e)$ and $A^k= (M_S(f,e))^k=M_S(f^k,e)$ (by Proposition \ref{P4.3}).
  By Proposition \ref{P.24} there exist  a reduced matrix factorization $(\phi_k , \psi_k)$ of $f$  and non-negative integers $t_k$ and $r_k$ such that $(A^k,A^{q-k})\sim (\phi_k , \psi_k) \oplus (f,1)^{t_k} \oplus (1,f)^{r_k}$. This gives  by Remark \ref{R3.8} (b), (c), and (d)  that  $\Cok_{S[[u,v]]}(A^k,A^{q-k})^{\maltese }= \Cok_{S[[u,v]]}(\phi_k , \psi_k)^{\maltese} \oplus [R^{\bigstar}]^{t_k+r_k} $.
  By Krull-Remak-Schmidt theorem (KRS)\cite[Corollary 1.10]{CMR},  there exist  non-negative integers $n_{(e,k,1)},\dots., n_{(e,k,t)}$ such that $\Cok_{S[\![u,v]\!]}(\phi_k , \psi_k)^{\maltese } \cong
   \bigoplus\limits_{j=1}^{t}M_j^{n_{(e,k,j)}}\cong\bigoplus\limits_{j=1}^{t}[\Cok_{S[\![u,v]\!]}(\alpha_j,\beta_j)^{\maltese}]^{ \oplus n_{(e,k,j)}}$\\$ \cong \Cok_{S[\![u,v]\!]}[\bigoplus_{j=1}^{t}(\alpha_j,\beta_j)^{ \oplus n_{(e,k,j)}}]^{\maltese}$. Now, from Proposition \ref{P3.15} and Proposition  \ref{A}  it follows that $ \Cok_{S}(\phi_k , \psi_k) \cong \Cok_{S}[\bigoplus_{j=1}^{t}(\alpha_j,\beta_j)^{ \oplus n_{(e,k,j)}}] \cong \bigoplus_{j=1}^{t} N_j^{ \oplus n_{(e,k,j)}}$ where $N_j= \Cok_S(\alpha_j,\beta_j)$ for all $j \in \{1,\dots,t\}$ . Therefore, $ F_*^e(S/f^kS) \cong \Cok_S(A^k,A^{q-k})=\Cok_S[(\phi_k , \psi_k) \oplus (f,1)^{r_k} \oplus (1,f)^{t_k}]=R^{r_k}\oplus\bigoplus_{j=1}^{t}N_j^{ \oplus n_{(e,k,j)}}$. This shows that $F_*^e(S/(f^k))$, for every $e\in \mathbb{N}$ and  $1 \leq k<p^e $, is  a direct sum with direct summands taken from $\{ R , N_1,\dots,N_t\}$.



Now suppose $F_*^e(S/(f^k))$  is  a direct sum with direct summands taken from indecomposable $R$-modules  $ N_1,\dots,N_t $ for every $e\in \mathbb{N}$ and  $1 \leq k<p^e $. Therefore, for each $1 \leq k \leq q-1$,  there exist  non-negative integers $n_{(e,k)},n_{(e,k,1)},\dots., n_{(e,k,t)}$ such that
\begin{equation}\label{E31}
 \Cok_S(A^k,A^{q-k}) \cong F_*^e(S/f^kS)\cong R^{\oplus n_{(e,k)}} \oplus \bigoplus\limits_{j=1}^{t}N_j^{\oplus n_{(e,k,j)}}  \text{ over }  R.
\end{equation}
Since $F_*^e(S/f^kS)$ is a MCM $R$-module  (by Corollary \ref{C4.12}), it follows that   $N_j$ is an indecomposable non-free MCM $R$-module for each $j \in \{1,\dots,t\}$ and hence
 by Proposition \ref{A} $N_j= \Cok_S(\alpha_j,\beta_j)$ for some reduced matrix factorization $(\alpha_j,\beta_j)$ for  all $j $. If  $M_j = \Cok_{S[\![u,v]\!]}(\alpha_j,\beta_j)^{\maltese}$, it follows that
 $ \Cok_{S[\![u,v]\!]}(A^k,A^{q-k})^{\maltese}= (R^{\bigstar})^{n_{(e,k)}} \oplus \bigoplus\limits_{j=1}^{t}M_j^{n_{(e,k,j)}} $ and hence by  (\ref{E13}) $R^{\bigstar}=S[\![u,v]\!]/(f+uv)$ has FFRT by $\{ R^{\bigstar},M_1,\dots, M_t\}$.

\end{proof}

The following result is a direct application of the above proposition
\begin{corollary}
Let $K$ be an algebraically closed field of prime characteristic $p > 2$ and $q=p^e$. Let  $S:= K[\![x_1,\dots,x_{n}]\!]$ and let $ \mathfrak{m}$ be the maximal ideal of $S$ and $f \in \mathfrak{m}^2 \setminus \{0\}$. Let $R=S/(f)$ and $R^{\bigstar}=S[\![u,v]\!]/(f+uv)$. If $R^{\bigstar}=S[\![u,v]\!]/(f+uv)$ has FFRT over $R^{\bigstar}$, then  $S/f^kS$ has  FFRT over $S/f^kS$ for every positive integer $k$.

\end{corollary}

The above corollary implies evidently the following.
\begin{corollary}

Let $K$ be an algebraically closed field of prime characteristic $p > 2$ and $q=p^e$. Let  $S:= K[\![x_1,\dots,x_{n}]\!]$ and let $ \mathfrak{m}$ be the maximal ideal of $S$ and $f \in \mathfrak{m}^2 \setminus \{0\}$. Let $R=S/(f)$ and $R^{\bigstar}=S[\![u,v]\!]/(f+uv)$. If $S/f^kS$ does not have  FFRT over $S/f^kS$ for some positive integer $k$, then $R^{\bigstar}$ does not have FFRT. In particular, if $R$ does not have  FFRT , then $R^{\bigstar}$ does not have FFRT.

\end{corollary}
An easy induction gives the following result.

\begin{corollary}
Let $K$ be an algebraically closed field of prime characteristic $p > 2$ and $q=p^e$. Let  $S:= K[\![x_1,\dots,x_{n}]\!]$ and let $ \mathfrak{m}$ be the maximal ideal of $S$,  $f \in \mathfrak{m}^2 \setminus \{0\}$ and let $R=S/(f)$. If $R$ does not have  FFRT, then the ring  $$S[\![u_1,v_1,u_2,v_2,\dots,u_t,v_t]\!]/(f+u_1v_1+u_2v_2+\dots+u_tv_t)$$ does not have FFRT for all $t\in \mathbb{N}$.
\end{corollary}


\section{A Class of rings that have FFRT but not finite CM type}
\label{Section: A Class of rings that have FFRT but not finite CM type}

We  keep the same notation as in section \ref{Section: The presentation as a cokernel of a Matrix Factorization} unless otherwise stated.  Recall that, a  local ring $(R, m)$ is said to have   finite CM representation type if there are only finitely many isomorphism classes
of indecomposable MCM R-modules. It is clear that every F-finite  local ring $(R, m)$ of prime characteristic that has finite CM representation type has also FFRT.  Smith and Van Den Bergh introduced in \cite{SV}  a class of rings that have FFRT but not finite CM representation type. However,    the main result of this section is to provide a new class of rings that have FFRT but not finite CM representation type Proposition \ref{P7.10}.

If $\alpha = (\alpha_1, \ldots , \alpha_n) \in \mathbb{Z_{+}}^n $, we write $ x^\alpha = x_1^{\alpha_1}\dots x_n^{\alpha_n}$ where $x_1,\dots,x_n$ are different variables.
\begin{lemma}\label{L4.25}
 Let  $f = x_1^{d_1}x_2^{d_2} \dots x_n^{d_n} $ be a monomial in $S$    where  $d_j \in \mathbb{Z_{+}}$ for each $j$. Let $\Gamma = \{(\alpha_1, \ldots , \alpha_n) \in \mathbb{Z_{+}}^n \, | \, 0 \leq \alpha_j\leq d_j \text{ for all  } 1\leq j \leq n \}$ , $d=(d_1,\dots,d_n)$, and let $e$ be a positive integer such that $q=p^e > \max \{ d_1 , \dots , d_n \}+1$. If $A=M_S(f,e)$, then for each $1 \leq k \leq q-1$ the matrix $A^k=M_S(f^k,e)$ is equivalent to diagonal matrix, $D$,  of size   $r_e \times r_e$ in which the diagonal entries are of the form $x^c$  where $c \in \Gamma$. Furthermore, if $c=(c_1,\dots,c_n) \in \Gamma $ and
\begin{equation*}
   \eta_k(c_j)= \begin{cases}q- | c_jq - kd_j |  \text { if } | c_jq - kd_j | < q \\
0 \text{ otherwise }
\end{cases}
\end{equation*}
 then
\begin{equation*}
  \Cok_S(A^q,A^{q-k}) = \bigoplus_{c \in \Gamma } \left[ \Cok_S(x^c, x^{d-c})\right]^{\oplus \eta_k(c)}
\end{equation*}
where $ \eta_k(c) = [K: K^q]\prod_{j=1}^n \eta_k(c_j) $ with the convention that $M^{\oplus0}=\{0\}$ for any module $M$ and $(x^c, x^{d-c})$ is the $1 \times 1$ matrix factorization of $f$.
\end{lemma}

\begin{proof}
Choose $e \in \mathbb{N}$ such that $q=p^e > \max \{ d_1 , \dots , d_n \}+1$ and let  $1\leq k \leq q-1 $. If $j=\lambda x_1^{\beta_1}\dots x_n^{\beta_n}\in \Delta_e$, we get $F_*^e(jf^k)=F_*^e(\lambda x_1^{kd_1+\beta_1}\dots x_n^{kd_n+\beta_n})$. Since $d_j , k \in \{ 0 , \dots , q-1 \} $,  there exist $ 0 \leq c_{i} \leq d_{i}  $ and $ 0 \leq u_{i}\leq q-1 $ for each $1 \leq i \leq n$ such that   $ d_i k  + \beta_{i} = c_{i}q + u_{i} $ and hence $ F_*^e(jf^k) = x_1^{c_{1}}\dots x_n^{c_{n}}F_*^e(\lambda x_1^{u_{1}}\dots x_n^{u_{n}}) $.
Therefore each  column and each  row of  $M_S(f^k,e)$ contains only one non-zero element of the form $ x_1^{c_{1}}\dots x_n^{c_{n}} $ where $  0 \leq c_{i}\leq  d  $ for all $1 \leq i \leq n$. Accordingly, using the row and column operations, the matrix $M_S(f^k,e)$ is equivalent to  a diagonal matrix, $D$,  of size   $r_e \times r_e$ in which the diagonal entries are of the form $ x_1^{c_1}\dots x_n^{c_n} $ where $  0 \leq c_{i}\leq  d  $ for all $1 \leq i \leq n$.  Now fix $c=(c_1,\dots,c_n) \in \Gamma$ and let $\eta(c)$ stand for  how many times $x^c$ appears as an element in the diagonal of $D$. It is obvious that $\eta(c)$  is exactly the same as the number of the $n$-tuples $(\alpha_1,\dots,\alpha_n)$ with $ 0 \leq \alpha_j \leq q-1 $   satisfying that
\begin{equation}\label{EEE2}
 F_*^e(\lambda x_1^{kd_1+\alpha_1}\dots x_n^{kd_n+\alpha_n}) = x_1^{c_1}\dots x_n^{c_n}F_*^e(\lambda x_1^{s_1}\dots x_n^{s_n})
\end{equation}
 where $ s_1,..,s_n \in \{0,\dots,q-1 \}$ for all $\lambda \in \Lambda_e$. However, an $n$-tuple $(\alpha_1,\dots,\alpha_n)$ with $ 0 \leq \alpha_j \leq q-1 $  will satisfy (\ref{EEE2}) if and only if  $\alpha_j= c_jq - kd_j + s_j$  for some $0 \leq s_j \leq q-1 $  whenever $1\leq j \leq n$. As a result,  the $n$-tuples $(\alpha_1,\dots,\alpha_n)\in \mathbb{Z}^n$ will satisfy (\ref{EEE2}) if and only if $|c_jq - kd_j|  < q$  for all $1\leq j \leq n$. Indeed, for $1\leq j \leq n$, set $u_j= c_jq - kd_j$. If $ 0 \leq u_j  < q$, we can choose $\alpha_j \in \{u_j, u_j+1,\dots,q-1 \}$. On the other hand, if    $ -q < u_j  < 0$, then $\alpha_j$ can be taken from $\{q-1+u_j, q-2+u_j,\dots,-u_j+u_j\}$. Therefore, if $0 \leq |c_jq - kd_j|  < q$, then $\alpha_j$ can be chosen by $q-|c_jq - kd_j|$ ways.
Set
\begin{equation*}
   \eta_k(c_j)= \begin{cases}q- |c_jq - kd_j | \text { if } | c_jq - kd_j | < q \\
0 \text{ otherwise }
\end{cases}
\end{equation*}
 Thus we get that $ \eta_k(c) = [K: K^q]\prod_{j=1}^n \eta_k(c_j) $. Now if $\tilde{\Gamma}= \{ c \in \Gamma \, | \, \eta_k(c)> 0 \}$, it follows that
\begin{equation*}
 \Cok_S(A^q,A^{q-k}) = \bigoplus_{c \in \tilde{\Gamma} } \left[ \Cok_S(x^c, x^{d-c})\right]^{\oplus \eta_k(c)}= \bigoplus_{c \in \Gamma } \left[ \Cok_S(x^c, x^{d-c})\right]^{\oplus \eta_k(c)}.
\end{equation*}

\end{proof}

\begin{corollary}\label{C32}
Let $K$ be an algebraically closed   field of prime characteristic $p > 2$ and $q=p^e$. Let  $S:= K[\![x_1,\dots,x_{n}]\!]$ and $f =x_1^{d_1}x_2^{d_2} \dots x_n^{d_n} $ where  $d_j \in \mathbb{N}$ for each $j$ . Then $R^{\bigstar}=S[\![u,v]\!]/(f+uv)$ has FFRT over $R^{\bigstar}$. Furthermore, for every $e \in \mathbf{N}$ with $q=p^e > max \{ d_1 , \dots , d_n \}+1$, $F_*^e(R^{\bigstar})$ has the following decomposition:
\begin{equation*}
 F_*^e(R^{\bigstar})= (R^{\bigstar})^{q^n} \bigoplus \bigoplus_{k=1}^{q-1}\left[\bigoplus_{c \in \Gamma } \left[\Cok_{S[\![u,v]\!]}(x^c,x^{d-c})^{\maltese}\right]^{\oplus \eta_k(c)} \right]
\end{equation*}
where $\eta_k(c)$ and $\Gamma$ as in the above lemma.
\end{corollary}
\begin{proof}
Let $e \in \mathbb{N}$ with $q=p^e > max \{ d_1 , \dots , d_n \}+1$ and let  $1 \leq k \leq q-1 $. Let $\Gamma $ and $ \eta_k(c)$ be as in the above lemma. If $A=M_S(f,e)$,   it follows that
\begin{equation*}
  F_*^e(S/f^k) \cong \Cok_S(A^k,A^{q-k}) \cong \bigoplus_{c \in \Gamma } \left[ \Cok_S(x^c, x^{d-c})\right]^{\oplus \eta_k(c)}.
\end{equation*}
 If $ \mathfrak{M}= \{\Cok_S(x^c, x^{d-c})\, | \, c \in \Gamma \}\cup \{ F^j(S/f^i)\, |\, p^j \leq \max \{ d_1 , \dots , d_n \}  \text{ and }   0 \leq i \leq p^j\} $,
then $F_*^e(S/(f^k))$  is  a direct sum with direct summands taken from the finite set $\mathfrak{M}$ for every $e\in \mathbb{N}$ and  $1 \leq k<p^e $. By Proposition \ref{P30}  $R^{\bigstar}$ has FFRT.

Furthermore, we can describe explicitly the direct summands of $F_*^e(R^{\bigstar})$. Indeed, if $\hat{\Gamma}:= \{ c \in \Gamma \, | \, \eta_k(c)> 0 \text{ and } c \notin \{d,0\} \} $, it follows that  $$(A^k,A^{q-k}) \sim \bigoplus_{c \in \hat{\Gamma} } (x^c,x^{d-c})^{\oplus \eta_k(c)}\bigoplus (x^d,1)^{\oplus \eta_k(d)}\bigoplus (1,x^d)^{\oplus \eta_k(0)} .$$
Recall  by Remark \ref{R3.8} that  $( A^k, A^{q-k} )^{\maltese}$ is a matrix factorization of $f+uv$ and
$$(A^k,A^{q-k})^{\maltese} \sim \bigoplus_{c \in \hat{\Gamma} } \left[(x^c,x^{d-c})^{\maltese}\right]^{\oplus \eta_k(c)}\bigoplus \left[(x^d,1)^{\maltese}\right]^{\oplus \eta_k(d)}\bigoplus \left[(1,x^d)^{\maltese}\right]^{\oplus \eta_k(0)}. $$
Therefore
\begin{eqnarray*}
 \Cok_{S[\![u,v]\!]}(A^k,A^{q-k})^{\maltese}   &=& \bigoplus_{c \in \hat{\Gamma} } \left[\Cok_{S[\![u,v]\!]}(x^c,x^{d-c})^{\maltese}\right]^{\oplus \eta_k(c)}\bigoplus \left[\Cok_{S[\![u,v]\!]}(x^d,1)^{\maltese}\right]^{\oplus \eta_k(d)}  \\
& & \bigoplus \left[\Cok_{S[\![u,v]\!]}(1,x^d)^{\maltese}\right]^{\oplus \eta_k(0)}.
\end{eqnarray*}

By  Proposition  \ref{P21},  the above equation, and the convention that $M^{\oplus 0}=\{0\}$, we can write
\begin{equation*}
 F_*^e(S[\![u,v]\!]/(f+uv)) = (R^{\bigstar})^{q^n} \bigoplus \bigoplus_{k=1}^{q-1}\left[\bigoplus_{c \in \Gamma } \left[\Cok_{S[\![u,v]\!]}(x^c,x^{d-c})^{\maltese}\right]^{\oplus \eta_k(c)} \right].
\end{equation*}

\end{proof}

We will take benefit from the proof of Corollary \ref{C32} above when we compute the $F$-signature in section 9.
The fact that the hypersurface in Corollary \ref{C32} has FFRT can be also proved differently in  the following proof.

\begin{proof}
If $f =x_1^{d_1}x_2^{d_2} \dots x_n^{d_n} $ where  $d_j \in \mathbb{N}$ for each $j$, notice that $f+uv$ is an irreducible polynomial in the ring $K[x_1,\cdots,x_n,u,v]$ and consequently the ideal $(f+uv)$  is a toric ideal and hence $K[x_1,\cdots,x_n,u,v]/(f+uv)$ is a affine toric ring (For more details on affine toric ideals and  affine toric rings see \cite[Section 2]{ES}, \cite[Section 1.2]{CLS}, and \cite[Section 2]{GHP}). Since an affine toric ring is a direct summand of a polynomial ring  \cite{Hoc}, it follows from \cite[Proposition 3.1.6]{SV} that the ring  $K[x_1,\cdots,x_n,u,v]/(f+uv)$ has FFRT and consequently we obtain that  $K[\![x_1,\cdots,x_n,u,v]\!]/(f+uv)$ has FFRT.
\end{proof}

\bigskip
If  $(S,\mathfrak{n})$ is a regular local ring, and  $R = S/( g )$, where
$0 \neq g \in \mathfrak{n}^2$, then  R is a simple  singularity (relative to the presentation $R = S/( g )$) provided there are only finitely
 many ideals $L$ of $S$ such that $g \in  L^2$ \cite[Definition 9.1]{CMR}.

\begin{proposition}\label{P7.10}
Let $K$ be an algebraically closed field with $char(k) >2$,
 and let $S = K[\![x_1, \dots, x_k]\!]$ where $k > 2$.
 If  $ f \in S$ is a monomial of degree greater than $3$ and  $R^{\bigstar}=S[\![u,v]\!]/(f+uv)$,  then $R^{\bigstar}$ has FFRT but it does not have finite CM representation type.
\end{proposition}
\begin{proof}
Let $t$ be the degree of the monomial $f$ and let $\mathfrak{m}$ be the maximal ideal of $S$. Clearly, $t$ is the largest natural number satisfying $ f \in \mathfrak{m}^t - \mathfrak{m}^{t+1}$ and consequently the multiplicity $e(R)$ of the ring $R$ is  $e(R)=t$ (by \cite[Corollary  A.24 page 435]{CMR} ). Since $e(R)=t > 3$ ,  $R$ is not a simple singularity   \cite[ Lemma  9.3]{CMR} . Therefore, by  \cite[Theorem 9.2]{CMR} $R$ does not have finite CM type. Consequently,  by Proposition \ref{P28}, $R^{\bigstar}$ does not have finite CM type as well. However,  Corollary \ref{C32} implies that  $R^{\bigstar}$ has FFRT.
\end{proof}


\section{The F-signature of $ \frac{S[[u,v]]}{f+uv}$ when $f$ is a monomial}
\label{Section: The F-signature of f+uv when f is a monomial}

We will keep the same notation as in notation \ref{N4.1} unless otherwise stated. \\

\begin{notation}\label{N1}
Let $\Delta = \{1,\dots,n \} $ and let  $d, d_1,\dots,d_n$ be real numbers.  For every $1\leq s \leq n-1$, define
\begin{equation*}
  W^{(n)}_s = \sum_{j_1,\dots,j_s \in \Delta } [(d-d_{j_1})\dots(d- d_{j_s})(\prod_{j \in \Delta \setminus \{j_1,\dots,j_s\}}d_j)]
\end{equation*}
\begin{equation*}
W^{(n)}_n= \prod_{i=1}^{n}(d-d_{i}) \text{  and  }  W^{(n)}_0= \prod_{i=1}^{n}d_{i}.
\end{equation*}

\end{notation}

According to the above notation,  we can observe the following remark
\begin{remark}\label{R5.3}
If $n\geq 2$, then  $W^{(n)}_j= (d-d_n)W^{(n-1)}_{j-1} + d_nW^{(n-1)}_{j}$ for every $ 1 \leq j \leq n-1 $.
\end{remark}

 The following lemma is needed to prove Proposition \ref{P8.4}
 \begin{lemma}\label{L5.4}
  If $r$ , $q$, $d_j$ and $u_j$ are real numbers for all $1 \leq j \leq n$, then

  \begin{equation*}
    \prod_{j=1}^{n}(d_jr + \frac{q(d-d_j)}{d}+ u_j) = \sum_{j=0}^{n} \frac{q^j}{d^j}W^{(n)}_jr^{n-j}+ \sum_{c=0}^{n-1}g^{(n)}_c(q)r^{c}
  \end{equation*}

 where $g^{(n)}_c(q)$ is a polynomial in $q$ of degree $n-1-c$ for all $0\leq c \leq n-1$.
 \end{lemma}

 \begin{proof}
 By induction on $n$, we will prove this lemma. It is clear when $n=1$. The induction hypothesis implies that

 \begin{equation*}
  \prod_{j=1}^{n+1}(d_jr + \frac{q(d-d_j)}{d}+ u_j)= A + B + C
 \end{equation*}
 where
 \begin{eqnarray*}
   A  &=&  d_{n+1}r \prod_{j=1}^{n}(d_jr + \frac{q(d-d_j)}{d}+ u_j)\\
      &=&  \sum_{j=0}^{n}d_{n+1} \frac{q^j}{d^j}W^{(n)}_jr^{n-j+1}+\sum_{c=0}^{n-1}d_{n+1}g^{(n)}_c(q)r^{c+1}\\
 B  &=& \frac{q(d-d_{n+1})}{d} \prod_{j=1}^{n}(d_jr + \frac{q(d-d_j)}{d}+ u_j)\\
      &=& \sum_{j=0}^{n}(d-d_{n+1}) \frac{q^{j+1}}{d^{j+1}}W^{(n)}_jr^{n-j} +\sum_{c=0}^{n-1}\frac{q(d-d_{n+1})}{d}g^{(n)}_c(q)r^{c}
 \end{eqnarray*}

 \begin{eqnarray*}
   C &=& u_{n+1}\prod_{j=1}^{n}(d_jr + \frac{q(d-d_j)}{d}+ u_j)  \\
     &=&  \sum_{j=0}^{n} u_{n+1}\frac{q^j}{d^j}W^{(n)}_jr^{n-j}+\sum_{c=0}^{n-1}u_{n+1}g^{(n)}_c(q)r^{c}.
 \end{eqnarray*}
If $D=\sum_{j=0}^{n}d_{n+1} \frac{q^j}{d^j}W^{(n)}_jr^{n-j+1}$ and $E=\sum_{j=0}^{n}(d-d_{n+1}) \frac{q^{j+1}}{d^{j+1}}W^{(n)}_jr^{n-j}$, it follows that
 \begin{eqnarray*}
  D+E &=& d_{n+1}W^{(n)}_0r^{n+1}+ \sum_{j=1}^{n}d_{n+1} \frac{q^j}{d^j}W^{(n)}_jr^{n-j+1}\\
   & & +\sum_{j=0}^{n-1}(d-d_{n+1}) \frac{q^{j+1}}{d^{j+1}}W^{(n)}_jr^{n-j}+ (d-d_{n+1}) \frac{q^{n+1}}{d^{n+1}}W^{(n)}_n \\
   &=& d_{n+1}W^{(n)}_0r^{n+1}+ \sum_{j=1}^{n}d_{n+1} \frac{q^j}{d^j}W^{(n)}_jr^{n-j+1}\\
   & & +\sum_{j=1}^{n}(d-d_{n+1}) \frac{q^{j}}{d^{j}}W^{(n)}_{j-1}r^{n-j+1}+ (d-d_{n+1}) \frac{q^{n+1}}{d^{n+1}}W^{(n)}_n \\
   &=& d_{n+1}W^{(n)}_0r^{n+1}+ \sum_{j=1}^{n}\frac{q^j}{d^j}[d_{n+1} W^{(n)}_j+(d-d_{n+1}) W^{(n)}_{j-1}]r^{n-j+1} \\
   & & +(d-d_{n+1}) \frac{q^{n+1}}{d^{n+1}}W^{(n)}_n.
\end{eqnarray*}

 Now apply Remark \ref{R5.3} to get that
\begin{equation}\label{E.8.14}
D+E= \sum_{j=0}^{n+1} \frac{q^j}{d^j}W^{(n+1)}_jr^{n+1-j}.
\end{equation}

Now define
 $$g^{(n+1)}_n(q)=d_{n+1}g^{(n)}_{n-1}(q)+u_{n+1}W^{(n)}_0,$$
$$g^{(n+1)}_0(q)=\frac{q}{d}(d-d_{n+1})g^{(n)}_0(q)+u_{n+1}\frac{q^n}{d^n}W^{(n)}_n+u_{n+1}g^{(n)}_0(q),$$
and
\begin{equation*}
 g^{(n+1)}_j(q)= d_{n+1}g^{(n)}_{j-1}(q)+ \frac{q(d-d_{n+1})}{d}g^{(n)}_{j}(q)+ u_{n+1}\frac{q^{n-j}}{d^{n-j}}W^{(n)}_{n-j}+u_{n+1}g^{(n)}_{j}(q)
\end{equation*}
 for every $1\leq j \leq n-1 $.

 Notice that $g^{(n+1)}_j(q)$ is a polynomial in $q$ of degree $n-j$ for all $0\leq j \leq n $ because $g^{(n)}_c(q)$ is a polynomial in $q$ of degree $n-1-c$ for all $0\leq c \leq n-1$.

It is easy to verify that
 \begin{equation*}
   \sum_{c=0}^{n-1}d_{n+1}g^{(n)}_c(q)r^{c+1}+ \sum_{c=0}^{n-1}\frac{q(d-d_{n+1})}{d}g^{(n)}_c(q)r^{c} + C = \sum_{c=0}^{n}g^{(n+1)}_c(q)r^{c},
 \end{equation*}
 and hence
 \begin{eqnarray*}
   \prod_{j=1}^{n+1}(d_jr + \frac{q(d-d_j)}{d}+ u_j) &=& A+B+C = E+D+ \sum_{c=0}^{n}g^{(n+1)}_c(q)r^{c} \\
   &=&   \sum_{j=0}^{n+1} \frac{q^j}{d^j}W^{(n+1)}_jr^{n+1-j}+ \sum_{c=0}^{n}g^{(n+1)}_c(q)r^{c}.
 \end{eqnarray*}

 \end{proof}






\begin{proposition}\label{P8.4}
 Let  $f= x_1^{d_1}\dots x_n^{d_n}$ be a monomial in $S=K[\![x_1,\dots,x_n]\!]$ where $d_j  $ is a positive  integer for each $1\leq j \leq n $. If $d = \max \{d_1,\dots, d_n \}$ and $R^{\bigstar}=S[\![u,v]\!]/(f+uv)$, then  the F-signature of $R^{\bigstar}$ is given by  \\
 \begin{equation}
 \mathbb{S}(R^{\bigstar})=\frac{2}{d^{n+1}}\left[\frac{d_1d_2\dots d_n}{n+1} + \frac{W^{(n)}_1}{n}+\dots+\frac{W^{(n)}_s}{n-s+1}+\dots+\frac{W^{(n)}_{n-1}}{2}\right]
\end{equation}
 where $W^{(n)}_1,\dots,W^{(n)}_{n-1}$ are defined as in the notation \ref{N1}. Therefore, $\mathbb{S}(R^{\bigstar})$ is positive.


\end{proposition}

\begin{proof}
 Let $R=S/fS$ and  $R^{\bigstar}=S[\![u,v]\!]/(f+uv)$. Set $[K:K^p]=b$ and recall from the notation \ref{N4.1} that $\Lambda_e$ is the basis of $K$ as $K^q$-vector space where $q=p^e$.  We know from Corollary \ref{C4.19}  that
 \begin{equation}\label{E1}
  \sharp ( F_*^e(R^{\bigstar}),R^{\bigstar})= r_e + 2 \sum_{k=1}^{q-1} \sharp ( \Cok_S(A^k),R)
 \end{equation}
 where $r_e=b^eq^n$ ,  $A=M_S(f,e)$ and $A^k=M_S(f^k,e)$. Since $f^k$ is a monomial, it follows from Lemma \ref{L4.25} that  the matrix $A^k=M_S(f^k,e)$ is equivalent to a diagonal matrix $D$ whose diagonal entries are taken from the set $\{x_1^{u_1}\dots x_n^{u_n} |0\leq u_j\leq d_j  \text{  for   all } 1\leq j \leq n  \} $.
This makes $\Cok_S(A^k)= \Cok_S(D)$ and consequently the number $\sharp ( \Cok_S(A^k),R) $ is exactly the same as the number of the $n$-tuples $(\alpha_1,\dots,\alpha_n)$ with $0 \leq \alpha_j \leq q-1 $  satisfying that
\begin{equation}\label{EE2}
 F_*^e(\lambda x_1^{kd_1+\alpha_1}\dots x_n^{kd_n+\alpha_n}) = x_1^{d_1}\dots x_n^{d_n}F_*^e(\lambda x_1^{s_1}\dots x_n^{s_n})
\end{equation}
  where $ s_1,..,s_n \in \{0,\dots,q-1 \}$ for all $\lambda \in \Lambda_e$. However, an $n$-tuple $(\alpha_1,\dots,\alpha_n)$ with $ 0\leq \alpha_j \leq q-1 $ will satisfy  (\ref{EE2}) if and only if  $\alpha_j= d_j(q-k) + s_j$  for some $ 0\leq s_j \leq q-1 $  for each $ 1 \leq j \leq n $. As a result,  the $n$-tuples $(\alpha_1,\dots,\alpha_n)\in \mathbb{Z}^n$ will satisfy  (\ref{EE2}) if and only if $d_j(q-k)\leq \alpha_j < q$  for all $ 1 \leq j \leq n $.
Set  $ N_j(k):= \{ \alpha_j \in \mathbb{Z} \, | \, d_j(q-k)\leq \alpha_j < q \}$ for all $ 1 \leq j \leq n $.  Therefore,
\begin{equation}\label{E18}
  \sharp ( \Cok_S(A^k),R)= b^e|N_1(k)||N_2(k)|\dots|N_n(k)|.
\end{equation}
Since  $|N_j(k)|= q-d_jq+d_jk$  for all $ 1 \leq j \leq n $, it follows that
\begin{equation}\label{EE18}
  \sharp ( \Cok_S(A^k),R) = b^e\prod_{j=1}^{n}(q-d_jq+d_jk).
\end{equation}

Let  $d = \max \{d_1,\dots, d_n \}$, then $\sharp ( \Cok_S(A^k),R)\neq 0 $ if and only if $ N_j(k)\neq 0$ for all $ 1 \leq j \leq n $ if and only if $\frac{q(d-1)}{d} < k$.
Let $q=du+t$ where $t\in \{0,..,d-1 \}$. If $t \neq 0$, then one can easily verify that
\begin{equation}\label{EE19}
 \frac{q(d-1)}{d} < q-\frac{q-t}{d} < \frac{q(d-1)}{d}+1.
\end{equation}
Therefore, \\ $\sharp ( \Cok_S(A^k),R)\neq 0$ if and only if $ k \in \{q-\frac{q-t}{d}+r \, | \, r \in \{0,\dots,\frac{q-t}{d}-1\} \}.$

However, if $t=0$,  it follows that $\frac{q(d-1)}{d}= q- \frac{q}{d}\in \mathbb{Z}$ and consequently $\sharp ( \Cok_S(A^k),R)\neq 0$ if and only if $k \in \{q-\frac{q}{d}+r \, | \, r \in \{1,\dots,\frac{q}{d}-1\} \}.$

First, assume  that  $t\neq 0$. This implies that
 \begin{eqnarray*}
   \sum_{k=1}^{q-1} \sharp ( \Cok_S(A^k),R) &=& b^e \sum_{k=q-\frac{q-t}{d}}^{q-1}\prod_{j=1}^{n}(q-d_jq+d_jk)\\
   &=& b^e \sum_{r=0}^{\frac{q-t}{d}-1}\prod_{j=1}^{n}(q-d_jq+d_j(r+q-\frac{q-t}{d})) \\
   &=& b^e\sum_{r=0}^{\frac{q-t}{d}-1}\prod_{j=1}^{n}(d_jr + \frac{q(d-d_j)}{d}+ \frac{d_jt}{d}).\\
 \end{eqnarray*}




Recall by Lemma \ref{L5.4} that

\begin{equation}\label{E61}
 \prod_{j=1}^{n}(d_jr + \frac{q(d-d_j)}{d}+ \frac{d_jt}{d})= \sum_{j=0}^{n} \frac{q^j}{d^j}W^{(n)}_jr^{n-j} +\sum_{c=0}^{n-1}g^{(n)}_c(q)r^{c}
\end{equation}

where  $g^{(n)}_c(q)$ is a polynomial in $q$ of degree $n-1-c$ for all $0 \leq c \leq n-1$.
Set $ \delta =\frac{q-t}{d}-1$.  By Faulhaber's formula \cite{JR}, if $s$ is a positive integer, we get the following polynomial in $\delta$ of degree $s+1$
\begin{equation}\label{E4}
 \sum_{r=1}^{\delta} r^s = \frac{1}{s+1}\sum_{j=0}^s(-1)^j\binom{s+1}{j}B_j\delta^{s+1-j}
\end{equation}
where $B_j$ are  Bernoulli numbers, $B_0=1$ and  $B_1= \frac{-1}{2}$. This makes
\begin{equation}\label{E62}
 \sum_{r=0}^{\delta} r^s=\frac{q^{s+1} }{(s+1)d^{s+1}}+V_s(q)
\end{equation}
where $V_s(q)$ is a polynomial of degree $s$ in $q$.
 From Faulhaber's formula and the  equations (\ref{E61}), and  (\ref{E62}),   we get that
 \begin{eqnarray*}
   \sum_{k=1}^{q-1} \sharp ( \Cok_S(A^k),R) &=& b^e\sum_{r=0}^{\delta}[\sum_{j=0}^{n} \frac{q^j}{d^j}W^{(n)}_jr^{n-j} +\sum_{c=0}^{n-1}g^{(n)}_c(q)r^{c}] \\
    &=& b^e[\sum_{j=0}^{n} \frac{q^j}{d^j}W^{(n)}_j\sum_{r=0}^{\delta}r^{n-j} +\sum_{c=0}^{n-1}g^{(n)}_c(q)\sum_{r=0}^{\delta}r^{c}]\\
    &=&b^e[\sum_{j=0}^{n} \frac{q^j}{d^j}W^{(n)}_j(\frac{q^{n-j+1} }{(n-j+1)d^{n-j+1}}+V_{n-j}(q)) \\
     & & + \sum_{c=0}^{n-1}g^{(n)}_c(q)(\frac{q^{c+1} }{(c+1)d^{c+1}}+V_c(q))]  \\
     &=& \frac{b^eq^{n+1}}{d^{n+1}}\sum_{j=0}^{n}\frac{W^{(n)}_j}{n-j+1}+ b^e[\sum_{j=0}^{n}\frac{q^j}{d^j}W^{(n)}_jV_{n-j}(q) \\
     & & +\sum_{c=0}^{n-1}g^{(n)}_c(q)\frac{q^{c+1} }{(c+1)d^{c+1}}+\sum_{c=0}^{n-1}g^{(n)}_c(q)V_c(q)].
 \end{eqnarray*}
Since $ \sum_{j=0}^{n}\frac{q^j}{d^j}W^{(n)}_jV_{n-j}(q)$ and $\sum_{c=0}^{n-1}g^{(n)}_c(q)\frac{q^{c+1} }{(c+1)d^{c+1}}+\sum_{c=0}^{n-1}g^{(n)}_c(q)V_c(q)$ are polynomials in $q=p^e$ of degree $n$ and $n-1$ respectively, it follows that

$$ \lim_{e\rightarrow \infty }\frac{1}{b^ep^{e(n+1)}}b^e[\sum_{j=0}^{n}V_{n-j}(q)+\sum_{c=0}^{n-1}g^{(n)}_c(q)\frac{q^{c+1} }{(c+1)d^{c+1}}+\sum_{c=0}^{n-1}g^{(n)}_c(q)V_c(q)]=0.$$
Therefore
\begin{eqnarray*}
\lim_{e\rightarrow \infty }\frac{1}{b^ep^{e(n+1)}}\sum_{k=1}^{q-1} \sharp ( \Cok_S(A^k),R) &=& \frac{1}{d^{n+1}}\sum_{j=0}^{n}\frac{W^{(n)}_j}{n-j+1}.
\end{eqnarray*}
By equation (\ref{E1}) and the above equation we conclude that the F-signature of the ring $R^{\bigstar}$ in the case that $t\neq 0$   is given by
\begin{equation}\label{E64}
 \mathbb{S}(R^{\bigstar})=\frac{2}{d^{n+1}}\Big[\frac{d_1d_2\dots d_n}{n+1} + \frac{W^{(n)}_1}{n}+\dots+\frac{W^{(n)}_s}{n-s+1}+\dots+\frac{W^{(n)}_{n-1}}{2} \Big].
\end{equation}

Second, assume that $t=0$ and hence   $q=du $.  Therefore,   $\frac{q(d-1)}{d}= q- \frac{q}{d}\in \mathbb{Z}$ and consequently
\begin{equation*}
\sharp ( \Cok_S(A^k),R)\neq 0 \Leftrightarrow   k \in \{q-\frac{q}{d}+r \, | \, r \in \{1,\dots,\frac{q}{d}-1\} \}.
\end{equation*}
 Therefore
\begin{eqnarray*}
   \sum_{k=1}^{q-1} \sharp ( \Cok_S(A^k),R) &=& b^e\sum_{k=q-\frac{q}{d}+1}^{q-1}\prod_{j=1}^{n}(q-d_jq+d_jk)\\
   &=& b^e \sum_{r=1}^{\frac{q}{d}-1}\prod_{j=1}^{n}(q-d_jq+d_j(r+q-\frac{q}{d})) \\
   &=& b^e \sum_{r=1}^{\frac{q}{d}-1}\prod_{j=1}^{n}(d_jr + \frac{q(d-d_j)}{d}).\\
 \end{eqnarray*}

By an argument similar to the above argument, we conclude the same result that

\begin{equation}\label{E64}
 \mathbb{S}(R^{\bigstar})=\frac{2}{d^{n+1}}\Big[\frac{d_1d_2\dots d_n}{n+1} + \frac{W^{(n)}_1}{n}+\dots+\frac{W^{(n)}_s}{n-s+1}+\dots+\frac{W^{(n)}_{n-1}}{2}\Big].
\end{equation}
\end{proof}

\section{The F-signature of $\frac{S[\![z]\!]}{(f+z^2)}$ when $f$ is a monomial}
\label{Section: The F-signature of f+z^2 when f is a monomial}

We will keep the same notation as in notation \ref{N4.1} unless otherwise stated. \\

\begin{proposition}
 Let $f= x_1^{d_1}\dots x_n^{d_n}$ be a monomial in $S=K[\![x_1,\dots,x_n]\!]$ where $d_j $ is a positive  integer for each $1 \leq j \leq n$ and $K$ is a field of prime characteristic $p>2$ with $[K:K^p] < \infty $. Let $R=S/fS$ and $R^{\sharp}=S[\![z]\!]/(f+z^2)$.  It follows that:  \\
1) If $ d_j=1 $ for each $1 \leq j \leq n$, then $ \mathbb{S}(S[\![z]\!]/(f+z^2))= \frac{1}{2^{n-1}}$. \\
2) If $d = \max \{d_1,\dots, d_n \}\geq 2$, then $ \mathbb{S}(S[\![z]\!]/(f+z^2))= 0$.
\end{proposition}

\begin{proof}
Set $[K:K^p]=b$ and recall from the notation \ref{N4.1} that $\Lambda_e$ is the basis of $K$ as $K^q$-vector space. We know  by  lemma \ref{L4.13} that
\begin{equation*}
 F_*^e(R^{\sharp})= \Cok_{S[\![z]\!]}\left[
                                              \begin{array}{cc}
                                                A^{\frac{q-1}{2}} & -zI \\
                                                zI &  A^{\frac{q+1}{2}} \\
                                              \end{array}
                                            \right]  \text{  where } A=M_S(f,e).
\end{equation*}
 It follows from Corollary \ref{C3.8} that
 \begin{equation}\label{E12}
  \sharp(F_*^e(R^{\sharp}),R^{\sharp})= \sharp(\Cok_S(A^{\frac{q-1}{2}}),R) + \sharp(\Cok_S(A^{\frac{q+1}{2}}),R).
 \end{equation}
Let $k\in \{\frac{q-1}{2}, \frac{q+1}{2}\}$ and
set  $ N_j(k):= \{ \alpha_j \in \mathbb{Z} \, | \, d_j(q-k)\leq \alpha_j < q \}$ for all $1\leq j \leq n $. Using the same argument that was previously used in the proof of Proposition \ref{P8.4}, it follows that
\begin{equation}\label{EE18}
  \sharp ( \Cok_S(A^k),R)= b^e|N_1(k)||N_2(k)|\dots|N_n(k)|=b^e\prod_{j=1}^{n}(q-d_jq+d_jk).
\end{equation}
Now if $d_1=d_2=\dots=d_n=1$, it follows from equation  (\ref{EE18}) that $ \sharp ( \Cok_S(A^k),R)= b^ek^n$ for $k\in \{\frac{q-1}{2}, \frac{q+1}{2}\}.$  Therefore,  the equation (\ref{E12}) implies that
\begin{equation}\label{EE15}
\sharp(F_*^e(R^{\sharp}),R^{\sharp})=b^e[ ( \frac{q-1}{2})^n + ( \frac{q+1}{2})^n]
\end{equation}
and consequently
\begin{equation*}
  \mathbb{S}(R^{\sharp})= \lim_{e\rightarrow \infty } \sharp(F_*^e(R^{\sharp}),R^{\sharp})/b^ep^{en} = \frac{1}{2^{n-1}}.
\end{equation*}
Now let  $d_i= \max \{d_1,\dots,d_n \}$  for some $1\leq i \leq n $.
First assume that $d_i=2 $.  If $k=\frac{q-1}{2}$ , it follows that $d_i(q-k) > q $ and consequently $|N_i(k)|=0 .$ The equation (\ref{EE18}) implies $\sharp ( \Cok_S(A^k),R)=0$. When $k=\frac{q+1}{2}$ , we get that  $d_i(q-k)=q-1$ and consequently $N_i(k)= \{q-1 \} $ which makes $|N_i(k)|=1 $. Notice that when $k=\frac{q+1}{2}$ and $d_j=1$ , it follows that $|N_j(k)|= \frac{q+1}{2} $. As a result, if $k=\frac{q+1}{2}$, we conclude that
\begin{equation*}
 \sharp ( \Cok_S(A^k),R)=b^e |N_1(k)||N_2(k)|\dots|N_n(k)| \leq b^e (\frac{q+1}{2})^{n-1}.
\end{equation*}
Therefore,

\begin{equation*}
  \sharp(F_*^e(R^{\sharp}),R^{\sharp})= \sharp(\Cok_S(A^{\frac{q-1}{2}}),R) + \sharp(\Cok_S(A^{\frac{q+1}{2}}),R) \leq b^e(\frac{q+1}{2})^{n-1}.
\end{equation*}

As a result,

\begin{equation*}
  \mathbb{S}(R^{\sharp})= \lim_{e\rightarrow \infty } \sharp(F_*^e(R^{\sharp}),R^{\sharp})/b^ep^{en} = 0.
\end{equation*}

Second  assume that $d_i > 2 $. In this case, for every $k\in \{\frac{q-1}{2}, \frac{q+1}{2}\}$, it follows that $d_i(q-k)> q$ and consequently $|N_i(k)|= 0 $. Therefore

\begin{equation*}
  \sharp(F_*^e(R^{\sharp}),R^{\sharp})= \sharp(\Cok_S(A^{\frac{q-1}{2}}),R) + \sharp(\Cok_S(A^{\frac{q+1}{2}}),R) = 0
\end{equation*}
 and consequently

 \begin{equation*}
  \mathbb{S}(R^{\sharp})= \lim_{e\rightarrow \infty } \sharp(F_*^e(R^{\sharp}),R^{\sharp})/b^ep^{en} = 0.
\end{equation*}

Notice that when $d = \max \{d_1,\dots, d_n \} > 2$, we can prove that $\mathbb{S}(R^{\sharp})=0$ using Fedder's Criterion \cite[Proposition 1.7]{RF}. Indeed, let $\mathfrak{m}$ be the maximal ideal of $S[\![z]\!]$ and let $R^{\sharp}=S[\![z]\!]/(f+z^2)$. If $d = \max \{d_1,\dots, d_n \} > 2$, then $(f+z^2)^{q-1}\in \mathfrak{m}^{[q]}$ which makes, by Fedder's Criterion,   $\sharp(F_*^e(R^{\sharp}),R^{\sharp})=0$ for all $e \in \mathbb{Z}^{+}$. This means clearly that
\begin{equation*}
 \mathbb{S}(R^{\sharp})= \lim_{e\rightarrow \infty } \sharp(F_*^e(R^{\sharp}),R^{\sharp})/b^ep^{en} = 0.
\end{equation*}

\end{proof}

\begin{center}
\textbf{Acknowledgment }
\end{center}

We would like  to thank the referee for carefully reading our manuscript and giving thoughtful comments and efforts  towards improving our manuscript.  We wish to thank also  Holger Brenner and Tom Bridgeland for their useful comments on the results in this paper.




\begin{thebibliography}{99}

\addcontentsline{toc}{section}{\protect\numberline{}{Bibliography}}
\bibitem[AL03]{AL}{\sc I. M. Aberbach and G. J. Leuschke}: \emph{The F-signature and strong F-regularity}, Math. Res. Lett. 10 (2003), 51–56.

\bibitem[BW13]{NW} {\sc N.R.~Baeth and R.~Wiegand}: \emph{Factorization theory and decomposition of modules},
Amer. Math. Monthly 120 (2013), no.~1, 3-–34.

\bibitem[Bru07]{WB} {\sc W. Bruns}: \emph{ Commutative Algebra Arising from the Anand–-Dumir–-Gupta Conjectures}, Proc.~Int.~Conf.~-– Commutative Algebra and Combinatorics No. 4, 2007, pp. 1-–38.

\bibitem[BG09]{WJ} {\sc W.~Bruns and J.~Gubeladze}: \emph{Polytopes, rings, and K-theory},  Springer 2009.


\bibitem[BH93]{BH} {\sc W. Bruns and J. Herzog}:  \emph{Cohen-Macaulay Rings}, vol. 39, Cambridge
studies in advanced mathematics, 1993.





\bibitem[CG96]{JR}{\sc J.H.~Conway and R.K.~Guy}: \emph{The Book of Numbers. New York}, NY: Copernicus, 1996. pp.~106-110.


\bibitem[CLS11]{CLS}{\sc D.A. ~Cox, J.B. ~Little and H.K. ~Schenck}: \emph{Toric varieties}, American Mathematical Society, U.S.A. (2011).


\bibitem[DQ16]{DaoQuyAssPrimes}
{\sc H.~Dao and H.~Q. Quy}: \emph{On the associated primes of local
  cohomology}, arXiv:1602.00421.

\bibitem[Eis80]{ED} {\sc D.~Eisenbud}: \emph{Homological algebra on a complete intersection, with an application to group representations},
Transactions of the AMS 260 (1980), 35–64.

\bibitem[EP15]{DI} {\sc D.~Eisenbud and I.~Peeva}: \emph{ Matrix factorizations for complete intersections and minimal
free resolutions}. arXiv:1306.2615

\bibitem[ES96]{ES} {\sc D.~Eisenbud and B.~Sturmfels}: \emph{ Binomial ideal}. Duke Math. J., 84, 1-45.


\bibitem[Fed83]{RF}{\sc R. Fedder}: \emph{ F-purity and rational singularity}, Trans. Amer. Math. Soc. 278 (1983) 461–480

\bibitem[GHP08]{GHP} {\sc V.~Gasharov, N.~Horwitz, and I.~Peeva}: \emph{ Hilbert functions over toric rings}. special volume of the Michigan Math. J. dedicated to Mel Hochster 57, 339–357 (2008)

\bibitem[Gil84]{RG} {\sc R.~Gilmer}: \emph{ Commutative Semigroup Rings}. U. Chicago Press (1984).

\bibitem[Hoc72]{Hoc}{\sc M.~Hochster}: \emph{Rings of invariants of tori, Cohen-Macaulay rings generated by monomials, and polytopes},
Ann. of Math. (2), 96:318–337, 1972.


\bibitem[HL02]{HL}{\sc C.~Huneke and G.~Leuschke}: \emph{Two theorems about maximal Cohen-Macaulay modules},
Math. Ann. 324 (2002), no.~2, 391–404.



\bibitem[KSSZ13]{MR3211813}
{\sc M.~Katzman, K.~Schwede, A.~K. Singh, and W.~Zhang}: \emph{Rings of
  {F}robenius operators}, Math. Proc. Cambridge Philos. Soc. \textbf{157}
  (2014), no.~1, 151--167.

\bibitem[LW12]{CMR} {\sc G.~Leuschke and R.~Wiegand}: \emph{Cohen-Macaulay Representations}, American Mathematical Society, Providence, RI, 2012.

\bibitem[LR74]{LR} {\sc L. S.~Levy and J. C.~Robson}:  \emph{Matrices and pairs of modules}, J. Algebra 29 (1974),
427-–454.



\bibitem[Mat86]{Mat} {\sc H.~Matsumura}: \emph{Commutative Ring Theory}, Cambridge Studies in Adv. Math. 8, Cambridge
University Press, Cambridge, 1986.



\bibitem[Sei97]{S} {\sc  G.~Seibert}: \emph{The Hilbert-Kunz function of rings of finite Cohen-Macaulay type}, Arch. Math.
69 (1997), 286-–296.

\bibitem[Shi11]{TS} {\sc  T.~Shibuta}:  \emph{One-dimensional rings of finite F-representation type}, Journal of Algebra, 15 April 2011, Vol.~332(1), pp.~434--441

\bibitem[SV97]{SV} {\sc K.~Smith,  and M.~Van den Bergh}:   \emph{Simplicity of rings of differential operators in prime
characteristic}, Proc. London Math. Soc. (3) 75 (1997), no.~1, 32–-62.



\bibitem[TT08]{TT}{\sc  S.~Takagi,and R.~Takahashi}: \emph{D-modules over rings with finite F-representation type Math}.
Res. Lett. 15 (2008), no.~3. 563-–581.

\bibitem[Tu12]{KT} {\sc K.~Tucker}: \emph{F-signature exists}, Invent. Math. 190 (2012), no.~3, 743–-765.

\bibitem[Yao05]{Y}{\sc Y.~Yao}:  \emph{Modules with finite F-representation type}, J. London Math. Soc. (2) 72 (2005), no.
1, 53–-72

\bibitem[Yao06]{Y2}{\sc Y.Yao}: \emph{Observations on the F-signature of local rings of characteristic p}, J.~Algebra 299 (2006),
no. 1, 198–-218.

\bibitem[Yosh90]{YY}{\sc Y.~Yoshino}: \emph{Cohen-Macaulay modules over Cohen-Macaulay rings},
London Mathematical Society Lecture Note Series, vol. 146,
Cambridge University Press, Cambridge, 1990.

\end{thebibliography}
\end{document}